\documentclass[1p]{elsarticle}
\usepackage{color}
\usepackage{enumitem}
\usepackage{textcomp}

\usepackage{amssymb}
\usepackage{amsthm}
\usepackage{amsbsy}

\newcommand{\RR}{{\mathbb{R}}}
\newcommand{\NN}{{\mathbb{N}}}
\newcommand{\ZZ}{{\mathbb{Z}}}
\newcommand{\CC}{{\mathbb{C}}}

\newcommand{\at}{\mbox{$\hat{a}$}}
\newcommand{\bt}{\mbox{$\hat{b}$}}
\newcommand{\ct}{\mbox{$\hat{c}$}}
\newcommand{\dt}{\mbox{$\hat{d}$}}

\newcommand{\Pt}{\hat{p}}
\newtheorem{remark}{Remark}
\newtheorem{theorem}{Theorem}
\newtheorem{lemma}{Lemma}
\newtheorem{example}{Example}
\journal{}

\begin{document}

\begin{frontmatter}


%\title{GaussMOP: a Matlab package computing simultaneous Gaussian quadrature rules for Multiple Orthogonal Polynomials\tnoteref{label1}}
%%% \tnotetext[label1]{}
 %\author{Teresa Laudadio\corref{cor1}\fnref{label2}}
 %\ead{teresa.laudadio@cnr.it}
 %\author{Nicola Mastronardi\corref{cor1}\fnref{label2}}
 %\ead{nicola.mastronardi@cnr.it}
 %\author{Walter Van Assche\corref{cor1}\fnref{label2}}
 %\ead{walter.vanassche@kuleuven.be}
 %\author{Paul Van Dooren\corref{cor1}\fnref{label2}}
 %\ead{paul.vandooren@uclouvain.be}
%%% \ead[url]{home page}
%%% \fntext[label2]{}

\title{A Matlab package computing simultaneous Gaussian quadrature rules for Multiple Orthogonal Polynomials} %\tnoteref{label1}}

 %\author{Teresa Laudadio\corref{cor1}\fnref{label2}}
 %\ead{teresa.laudadio@cnr.it}
 %\author{Nicola Mastronardi\corref{cor1}\fnref{label2}}
 %\ead{nicola.mastronardi@cnr.it}
 %\author{Walter Van Assche\corref{cor1}\fnref{label2}}
 %\ead{walter.vanassche@kuleuven.be}
 %\author{Paul Van Dooren\corref{cor1}\fnref{label2}}
 %\ead{paul.vandooren@uclouvain.be}
\author[IAC]{Teresa Laudadio\corref{cor1}}
\ead{teresa.laudadio@cnr.it}
\cortext[cor1]{Corresponding author}
\affiliation[IAC]{organization={Istituto per le Applicazioni del Calcolo ``Mauro Picone'', CNR},%Department and Organization
            addressline={Via Amendola 122/D}, 
            city={Bari},
            postcode={70126}, 
            country={Italy}}
						
			\author[IAC]{Nicola Mastronardi}
		%	\ead{nicola.mastronardi@cnr.it}
%\affiliation{organization={Istituto per le Applicazioni del Calcolo ``Mauro Picone'', CNR},%Department and Organization
            %addressline={Via Amendola 122/D}, 
            %city={Bari},
            %postcode={70126}, 
            %state={},
            %country={Italy}}	
						
			\author[KUL]{Walter Van Assche}
\affiliation[KUL]{organization={Department of Mathematics,
KU Leuven},%Department and Organization
            addressline={Celestijnenlaan 200B}, 
            city={Leuven},
            postcode={3001}, 
            country={Belgium}}
					
			\author[UCL]{Paul Van Dooren}
\affiliation[UCL]{organization={Department of Mathematical Engineering, Catholic University of Louvain},%Department and Organization
            addressline={Batiment Euler (A.202),
Avenue Georges Lemaitre 4}, 
            city={Louvain-la-Neuve},
            postcode={1348}, 
            country={Belgium}}													

\begin{abstract}
The aim of this  paper is to describe a Matlab package for  computing the simultaneous Gaussian quadrature rules associated with a variety of  multiple orthogonal polynomials. %  for $ r=2.$

Multiple orthogonal polynomials can be considered  as a generalization of  classical orthogonal polynomials, satisfying orthogonality constraints with respect to $ r$  different measures, with  $ r \ge 1.$  Moreover, they satisfy  $(r+2)$--term recurrence relations. In this manuscript, without loss of generality,   $ r$ is considered equal to $ 2.$ 
The so--called simultaneous Gaussian quadrature rules associated  with multiple orthogonal polynomials can be computed by solving a banded lower Hessenberg eigenvalue problem.
Unfortunately, computing the eigendecomposition of such a 
matrix turns out to be strongly ill--conditioned and the   {\tt Matlab} function {\tt balance.m} does not improve the condition of the eigenvalue problem. 
Therefore, most procedures for computing simultaneous Gaussian quadrature rules are implemented with variable precision arithmetic.
Here, we  propose a {\tt Matlab} package that  allows to reliably compute the  simultaneous Gaussian quadrature rules  in floating point arithmetic. It  makes use of a variant  of a new balancing procedure, recently  developed  by the authors of the present manuscript, that drastically reduces the condition of the  Hessenberg eigenvalue problem.
\end{abstract}




\end{frontmatter}


%\section{}
%\label{}

\section{Introduction}\label{sect:intro}

In this  paper, we consider the computation of simultaneous Gaussian quadrature rules associated with a variety of  multiple orthogonal polynomials (MOPs). %  for $ r=2.$
%%%%%%%%%%%%%%%%
MOPs originally appeared in Hermite-Pad\'e approximation (simultaneous
rational approximation) and number theory. Recently, they turned out to be very useful in  random matrix theory
\cite{Kuijlaars}, combinatorics \cite{Sokal} and Markov chains \cite{Manas}. Simultaneous Gaussian quadrature was
introduced in \cite{Borges94} to model computer graphics illumination, where the computation of different weighted integrals
with the same integrand function was needed. The aim was to minimize the evaluations of the integrand function and
maximize the order of the quadrature rules based on the same set of nodes. 
Simultaneous Gaussian quadrature rules associated with MOPs related to the modified Bessel functions of the first and
second kind were proposed in \cite{VanAssche2023}. Gaussian quadrature with these special weight functions (and also  with
hypergeometric or confluent hypergeometric weights and the exponential integral) requires the computation of the recurrence coefficients
of the corresponding orthogonal polynomials from the moments, which is an ill-conditioned numerical problem
(see, e.g., \cite{Gautschi1, Gautschi2}). Surprisingly, the recurrence coefficients of  MOPs
for such weights are explicitly known, so that their  numerical computation is avoided and
the quadrature formula can be computed by solving an eigenvalue problem obtained by properly arranging these recurrence coefficients into an  Hessenberg matrix, %with these recurrence coefficients usingeigenvalues and eigenvectors,
as proposed in \cite{VanAssche2005} and \cite{VanAssche2023}.
%%%%%%%%%%%%%%%%
%%%%%%%%%%%%%%%%%%%%%%%%%%%%%%%%%%%%%%%%%%%%%%%%%%%%%%%%%%%%%%%%%%%%%%%%%%%%

%
%%%%%%%%%%%%%%%%%%%%%%%%%%%%%%%%%%%%%%%%%%%%%%%%%%%%%%%%%%%%%%%%%%%%%%%%%%%%%
%The detailed knowledge of multiple orthogonal polynomials with respect to (generalized)
%hypergeometric functions has applications in random matrix theory, combinatorics, description
%of rational solutions to nonlinear differential-difference equations, such as Painlevé equations,
%number theory, among other fields. For instance, the analysis of singular values of products of
%Ginibre matrices in Refs 1,2 uses multiple orthogonal polynomials associated with weight functions
%expressed in terms of Meijer G-functions, a class of weights to which the weight (1) belongs.
%Furthermore, these polynomials are linked with the branched continued fractions introduced in3
%as the generating functions of $m $--Dyck paths, for the purpose of solving total-positivity problems
%involving combinatorially interesting sequences of polynomials. This connection, which leads to
%new results in both fields involved, will be further explored in forthcoming work. The research
%presented here enlarges and fits into studies on multiple orthogonal polynomials with respect
%to weights satisfying second-order differential equations, such as the modified Bessel function
%of second-order $K_{\nu}$ (see Refs 4,5), the modified Bessel function of first kind $I_{\nu}$ (see Refs 6,7), the
%confluent hypergeometric function (see 8).

%%%%%%%%%%%%%%%%%%%%%%%%%%%%%%%%%%%%%%%%%%%%%%%%%%%%%%%%%%%%%%%%%%%%%%%%%%%%

MOPs  are a generalization of orthogonal polynomials and can be divided into two classes: type I and type II \cite{VanAssche2005}. In this paper we focus on MOPs of type II.
Suppose $ r $ weight functions $w^{(i)}(x) \ge 0,$ with support $ \Delta^{(i)}, \; i=1,\ldots,r,$ on the real line   are given. Then, the sequence of  MOPs $ \left\{p_n(x)\right\}_{ n=0}^{\infty} $  of type II   satisfy the following orthogonality conditions \cite{VanAssche2005}:
\begin{equation}\label{eq:exactp}
\int_{\Delta^{(i)}} p_{n}(x) x^k  w^{(i)}(x)dx  =  0, \quad 0\le k \le n_{i}-1,
\end{equation}
with $ n= \sum_{i=1}^r n_i. $

Let $ \Delta = \bigcup_{i=1}^{r}\Delta^{(i)}. $ 
Two different systems of  MOPs of type II can be considered \cite{VanAssche2001,ismail05}:
\begin{enumerate}
\item
Angelesco system, where the  open intervals  $ \Delta^{(i)}$, $\; i=1,\ldots,r,$ are disjoint, i.e.,  $ \Delta^{(i)}\bigcap  \Delta^{(j)}= \emptyset$, for $ i \ne j,$ and the closed intervals $ \Delta^{(i)}$ are allowed to touch.
\item
algebraic Chebyshev system (AT system), where  %the  intervals  $ \Delta^{(i)}$,  $\; i=1,\ldots,r,$ are the same, i.e., 
$ \Delta^{(i)} =\Delta, \; i=1, \ldots,r. $ 

\end{enumerate}
%Let $ \Delta = \bigcup_{i=1}^{r}\Delta^{(i)}. $ 
Then, $ p_n(x)$ of type II  has exactly $ n $ zeros in  $ \Delta$ \cite[Th.~2]{VanAssche2001}.

A set of MOPs satisfies an $(r+2)$--term recurrence relation.
Without loss of generality, in this paper we focus on the case $r=2.$ Therefore, the set of MOPs satisfies a $4$--term recurrence relation\footnote{All the MOPS considered in the literature are monic, i.e., $ a_i =1,\; i=0,1,\ldots, n-1. $}
\begin{equation} \label{eq:MOP} 
xp_{i}(x) = a_i p_{i+1}(x)  + b_i p_{i}(x)  + c_i p_{i-1}(x)  + d_i p_{i-2}(x), \;\; i=0, \ldots, n-1,
\end{equation}
with $p_{-2}(x)=p_{-1}(x)=0.$ 
Writing (\ref{eq:MOP}) in matrix form, we obtain
$$
H_n 
\left[\begin{array}{c} 
p_0(x) \\ p_1 (x) \\ \vdots \\  p_{n-1}(x)
\end{array}\right] +
 a_{n-1}
\left[\begin{array}{c} 
0 \\ \vdots \\ 0 \\p_n (x) 
\end{array}\right] =
x
\left[\begin{array}{c} 
p_0(x) \\ p_1 (x) \\ \vdots \\  p_{n-1}(x)
\end{array}\right],
$$
where $  H_n $ is 
%
%Arranging (\ref{eq:MOP}) in matrix form, for $i=0, \ldots, n-1,$
 the $n\times n$ banded lower Hessenberg matrix with 2 sub--diagonals and one upper-diagonal~: 
\begin{equation} \label{H}
H_n := \left[\begin{array} {ccccccc} b_0 & a_0 & 0 & 0 & 0 &  \ldots & 0 \\
c_1 & b_1 & a_1 & 0 & 0 &  \ldots & 0 \\
d_2 & c_2 & b_2 & a_2 & 0 &  \ldots & 0 \\
0 & d_3 & c_3 & b_3 & a_3 &   \ddots & 0 \\
\vdots & \ddots  & \ddots   & \ddots & \ddots & \ddots & 0  \\
0 & \ldots &  0 & d_{n-2} & c_{n-2} & b_{n-2} & a_{n-2} \\
0 & \ldots &  0 & 0 & d_{n-1} & c_{n-1} & b_{n-1}  \\
\end{array} \right].
\end{equation}
%\begin{equation} \label{H}
%H_n := \left[\begin{array} {cccccccc} b_0 & a_0 & 0 & 0 & 0 & 0 & \ldots & 0 \\
%c_1 & b_1 & a_1 & 0 & 0 & 0 & \ldots & 0 \\
%d_2 & c_2 & b_2 & a_2 & 0 & 0 & \ldots & 0 \\
%0 & d_3 & c_3 & b_3 & a_3 & 0 & \ldots & 0 \\
%0 & 0 & d_4 & c_4 & b_4 & a_4 & \ldots & 0 \\
%\vdots &  &  & \ddots & \ddots & \ddots & \ddots &  \\
%0 & \ldots & 0 & 0 & d_{n-2} & c_{n-2} & b_{n-2} & a_{n-2} \\
%0 & \ldots & 0 & 0 & 0 & d_{n-1} & c_{n-1} & b_{n-1}  \\
%\end{array} \right].
%\end{equation}
The Gaussian quadrature rule associated with classical orthogonal polynomials
 can be retrieved from the eigenvalue decomposition of a symmetric tridiagonal matrix, and it is exact for polynomials of degree $2 n-1$ \cite{GolubW69}.
The theory of  simultaneous Gaussian quadrature rules  for the general case $ r >2 $ is described in  \cite{VanAssche2005}.

%It turns out that 
For $ r=2 $,  the simultaneous Gaussian quadrature rule associated with MOPs \cite{VanAssche2005},
$$ %\begin{equation}\label{eq:exactp}
\sum_{k=1}^{n}\omega_{k}^{(j)} f(x_j) =\int_{\Delta^{(i)}} f(x)   w^{(j)}(x)dx  +E_{n}^{(j)}(f),  \quad j=1,2,
$$ % \end{equation}
with $ E_{n}^{(j)}(f)=0$, if $ f $ is a polynomial of degree $ n+n_j-1, $  
can be retrieved from the eigenvalue decomposition of the matrix $ H_n $ \cite[Th.~3.2]{VanAssche2005},\cite[Th.~2]{VanAssche2023}. 
Therefore, the degree of exactness of the  simultaneous Gaussian quadrature rule  is maximal if all
 $n_j,\; j=1, 2, $ in (\ref{eq:exactp}) are equal \cite{LubAsc}. 
%Hence from now on we will use $ n_j:=m, \;i=1,\ldots,r, $ i.e., $n= rm. $  
%It is exact for polynomials of degree $\lfloor \frac{3}{2} n\rfloor-1$  \cite{Borges94,VanAssche2005,VanAssche2023}, where  $\lfloor \frac{3}{2} n\rfloor  $ denotes  the largest integer not exceeding $  \frac{3}{2} n.$ 
 The following theorem holds:
\begin{theorem}\label{th:quads}
The  nodes $x_{j}, \; j=1,\ldots,n, $ of the simultaneous Gaussian quadrature rule  are given by the eigenvalues of the banded Hessenberg matrix $ H_n$ (\ref{H}).
Moreover,  let us denote by 
 $ \boldsymbol{u}^{(j)}=[u_1^{(j)}, u_2^{(j)},\ldots, u_n^{(j)} ]^T, $  and $\boldsymbol{v}^{(j)}=[v_1^{(j)}, v_2^{(j)},\ldots, v_n^{(j)} ]^T $
 the  left and right eigenvectors of $ H_n $ associated with $x_{j}, $ respectively. Then,
\begin{equation}\label{eq:weights}
%\begin{array}{l}
{\displaystyle \omega_j^{(1)}}=\frac{\displaystyle  v_1^{(j)} f_{1,1}u_1^{(j)}}{\displaystyle \boldsymbol{u}^{(j)^T} \boldsymbol{v}^{(j)}},%\\
\quad 
{\displaystyle \omega_j^{(2)}}=\frac{\displaystyle v_1^{(j)}\left(  f_{2,1}u_1^{(j)}+f_{2,2}u_2^{(j)}\right)}{\displaystyle  \boldsymbol{u}^{(j)^T} \boldsymbol{v}^{(j)}},
%\end{array}
\quad j=1,\ldots,n,
\end{equation}
where  %$f_{i,k},\; i,k=1,2,$ are 
\begin{equation}\label{eq:D}
\begin{array}{ll}
f_{1,1}={\displaystyle \int_{\Delta^{(1)}}} p_0(x) w^{(1)}(x)dx, & \\ %{\color{red}f_{1,2}=0,} \\
&\\
f_{2,1}= {\displaystyle \int_{\Delta^{(2)}}} p_0(x) w^{(2)}(x)dx, & f_{2,2}= {\displaystyle \int_{\Delta^{(2)}}} p_1(x)w^{(2)}(x)dx.
\end{array}
\end{equation}
\end{theorem}
Hence,  simultaneous Gaussian quadrature rules associated with MOPs reduces to the computation of the  eigendecomposition of the Hessenberg matrix $ H_n,$ which,
unfortunately,  turns out to be strongly ill--conditioned \cite{VanAssche2023}.  Furthermore, the   {\tt Matlab} function {\tt balance.m} applied to the Hessenberg matrix $ H_n$  does not improve the condition of the eigenvalue problem  \cite{TNP2023} and, then, the {\tt Matlab} function  {\tt eig.m} yields unreliable results.
Therefore,  procedures for computing simultaneous Gaussian quadrature rules are implemented with variable precision arithmetic  \cite{VanAssche2023}.

%%%%%

Recently, 
simultaneous Gaussian quadrature rules have been proposed for MOPs associated with modified Bessel functions of the first and second kind \cite{VanAssche2023,TNP2023}, where the banded lower Hessenberg matrix $ H_n $ is totally nonnegative.
In particular, in \cite{TNP2023}, a new balancing procedure has been proposed   that drastically reduces the condition of the aforementioned Hessenberg eigenvalue problem, thereby allowing to compute the associated  simultaneous Gaussian quadrature rule  in floating point arithmetic in a reliable way.

Based on the results described in  \cite{TNP2023}, we develop here an algorithm for computing   simultaneous Gaussian quadrature rules  associated with different kinds of MOPs, for which the  banded  lower Hessenberg matrix  was not  totally nonnegative, and  we describe the associated 
  {\tt Matlab} package, which requires  only  $\mathcal{O}(n^2)$ computational  complexity and $\mathcal{O}(n)$ memory.

The paper is organized as follows. Notations  are introduced in Section~\ref{sect:not}. 
The handled classes of MOPs   are listed in Section~\ref{sect:MOP}. Moreover, the use of the {\tt Matlab} function {\tt ClassMOP.m}, generating the coefficients of the recurrence relations of the associated MOPs, is described in Section~\ref{sect:class}.
The use of the function {\tt GaussMOP.m}, computing the nodes and the weights of the chosen class of MOPs, is reported in Section~\ref{sect:alg}, followed by the description of the proposed
   numerical method  in Sections~\ref{subs:1} and \ref{subs:2}.
Numerical tests are reported in  Section~\ref{sect:NE}, followed by the concluding remarks. Finally, the {\tt Matlab} codes can be found in the Appendix.%\ref{sect:append}. 

\section{Notations}\label{sect:not}
Upper--case letters  $A, B, \ldots,   $ denote matrices and $ A_{m,n},$ or simply  $ A_{m}$ if $ m=n,$ denotes matrices of size $ (m,n). $ The entry $(i,j)$ of a matrix $ A $ is denoted by $ a_{i,j}.$ Submatrices are denoted  by the  colon notation of {\tt Matlab}, i.e.,   $ A(i : j,k : l)$ is the submatrix of $A$ obtained by the intersection of rows $ i$  to $ j $  and  columns $k$ to $l$, and  $A(i : j, :)$ and   $ A(: ,k : l)$  are  the rows of $ A$ from $ i$  to $ j $ and   the columns  of $ A $  from  $k$ to $l,$ respectively. 
\\
Given $ A \in \RR^{n \times n} $ and $ k \in \ZZ,$  $ -n+1 \le k \le n-1, $   $\mbox{\tt triu}(A,k) $ denotes the matrix with elements   on and above the $k$th diagonal of $ A.$
\\
Bold lower--case letters $\boldsymbol{x},\boldsymbol{y},  \ldots, \boldsymbol{\omega},  \ldots,$  denote vectors, and $ x_i $ denotes
the $i$th element of the vector $ \boldsymbol{x}$. \\
Lower--case letters $ x, y, \ldots, \lambda,\theta, \ldots, $ denote   scalars.
\\
The identity matrix of order $ n$ is denoted by $ I_n, $ and its $ i$th column,  $i=1,\ldots,n, $ i.e.,
the $i$th vector of the canonical basis of $ \RR^{n}, $  is denoted by $ \boldsymbol{e}_i$.
\\
The zero vector of  length  $ n$ is denoted by $\boldsymbol{o}_{n}. $\\
 The $i$--th subdiagonal of a matrix $ H \in \RR^{m \times n}$ is denoted by ${\tt diag}(H,-i). $\\ The diagonal matrix with entries $ d_1,\ldots, d_n $ is denoted by $\mbox{\tt diag}(d_1,\ldots,d_n). $ \\
The notation $\lfloor y \rfloor $ stands for the largest integer not exceeding $ y \in \RR_{+}.$ \\
The notation $   k \gg 0 $ stands for $ k \in \NN, $ with  $ k$ very large. \\
 If $ x \in \CC, $ $\mathcal{R} (x)$ denotes the real part of $ x. $ \\
Numbers in scientific notation $ a \times 10^{b},$ with $ a, b \in \RR, $  are represented as $ a(b) $ in  Section~\ref{sect:NE}.\\
A {\em flop} denotes a floating point operation (sum, subtraction, multiplication, division).   The square root is considered  a flop as well.

%\section{\ref{en:1a} {\color{blue} Classes of the multiple orthogonal polynomial}
%\section{\ref{en:1a} {\color{blue} Classes of the multiple orthogonal polynomials}}
\section{Classes of the multiple orthogonal polynomials}
\label{sect:MOP} 
The considered classes of MOPs  are listed below.  For each of them, the corresponding weights $  w^{(1)}(x)$ and $ w^{(2)}(x) $ , the integration intervals, the recurrence relations, and the coefficients (\ref{eq:D}), involved in the computation of the vectors of  weights $\boldsymbol{ \omega}^{(1)}$ and $\boldsymbol{ \omega}^{(2)}$,  are reported. %described for each case.

%$begin{enumerate}[label={\bf MOP\arabic* \\ \hspace*{-2cm}}]
%%%%%%%%%%%%%%%%%%%%%%%%%%%%%%%%%%%%%%%%%%%%%%%%%%%%%%%%%%%%%%%%%%%%%%%%%%%%%%%%%%%%%%%%%%%%%%%%
%\subSection
\bigskip

\noindent
$\mbox{\bf MOP}_1:$
{ Multiple Jacobi--Pi\~neiro polynomials} %$\mbox{\bf MOP}_1$
\\%$ $\\
\noindent
Weights: %$ x^{\alpha_j}e^{-x},\; \alpha_j >-1, \; j=1,\ldots, r.$\\
$$\left(  w^{(1)}(x), w^{(2)}(x) \right):= \left(x^{\alpha_1}(1-x)^{\alpha_0}, x^{\alpha_2}(1-x)^{\alpha_0}\right), $$
$ \alpha_j >-1,\; j=,0,1,2,  \;\; \alpha_1-\alpha_2 \notin \mathbb{Z}.$\\

\noindent
Interval: $\Delta^{(i)}=[0,1],\; i=1,2.$ \\

\noindent
Recurrence relation coefficients:
$$ \begin{array}{@{}l@{\hspace{.4mm}}l@{\hspace{.4mm}}l} 
 b_0&=(1+\alpha_1)/(2+\alpha_0+\alpha_1);\\
\mbox{\tt for}\;&i=1,2,\ldots\\
 b_{2i}&=\Bigl(36i^4+(48\alpha_0+28\alpha_1+20\alpha_2+38)i^3+ (21\alpha_0^2+8\alpha_1^2+4\alpha_2^2+30\alpha_0\alpha_1 \Bigr. \\
 & +18\alpha_0\alpha_2+15\alpha_1\alpha_2+39\alpha_0  %\Bigr.  \hphantom{xxxxxxxxxxxxxxxxxx}\\
 +19\alpha_1+19\alpha_2+9)i^2  +(3\alpha_0^3+10\alpha_0^2\alpha_1 \\ & +4\alpha_0^2\alpha_2+6\alpha_0\alpha_1^2+2\alpha_0\alpha_2^2+11\alpha_0\alpha_1\alpha_2+5\alpha_1^2\alpha_2+3\alpha_1\alpha_2^2
		+12\alpha_0^2 +3\alpha_1^2\\ 
		& +3\alpha_2^2+13\alpha_0\alpha_1	+   13\alpha_0\alpha_2+8\alpha_1\alpha_2+6\alpha_0+3\alpha_1+3\alpha_2)i	+  \alpha_0^2 +\alpha_0\alpha_1 \hphantom{nicola}\\% & i=1,2,\ldots\\
		& +\alpha_2\alpha_1^2 +2\alpha_2\alpha_1^2\alpha_0
	+2\alpha_0^2\alpha_1+\alpha_1^2\alpha_0 	+\alpha_2^2\alpha_0+\alpha_2^2\alpha_1+\alpha_0^3\alpha_1 +\alpha_0^2\alpha_1^2 \\
	& +\Bigl.\alpha_2^2\alpha_0\alpha_1+\alpha_2^2\alpha_1^2+2\alpha_2\alpha_0^2\alpha_1 +3\alpha_2\alpha_1\alpha_0+2\alpha_2\alpha_0^2 +\alpha_1\alpha_2+\alpha_0^3+\alpha_0\alpha_2 \Bigr)\\
  & \times
\Bigl((3i+\alpha_0+\alpha_2)(3i+\alpha_0+\alpha_1)(3i+\alpha_0+\alpha_2+1)(3i+\alpha_0+\alpha_1+2)\Bigr)^{-1};\\
%(3i+\alpha_0+\alpha_2)^{-1}(3i+\alpha_0+\alpha_1)^{-1}(3i+\alpha_0+\alpha_2+1)^{-1}(3i+\alpha_0+\alpha_1+2)^{-1};\\
\end{array}
$$
$$ 
\begin{array}{@{}l@{\hspace{.4mm}}l@{\hspace{-8.4mm}}l} 
    %%%%%%%%%%%%%%%%%%%%%%%%%%%%%%%%%%%%%%%%%
		\mbox{\tt for}\;&i=0,1,\ldots\\
    b_{2i+1}&=\Bigl(36i^4+(48\alpha_0+20\alpha_1+28\alpha_2+106)i^3+ 
        (21\alpha_0^2+4\alpha_1^2+8\alpha_2^2+18\alpha_0\alpha_1 \Bigr. \\ 
				&+30\alpha_0\alpha_2+15\alpha_1\alpha_2+105\alpha_0+41\alpha_1+65\alpha_2+111)i^2  +(3\alpha_0^3+4\alpha_0^2\alpha_1\\
				&+10\alpha_0^2\alpha_2+2\alpha_0\alpha_1^2+6\alpha_0\alpha_2^2+11\alpha_0\alpha_1\alpha_2+3\alpha_1^2\alpha_2+5\alpha_1\alpha_2^2+30\alpha_0^2+5\alpha_1^2\\
							&+13\alpha_2^2+23\alpha_0\alpha_1+47\alpha_0\alpha_2+22\alpha_1\alpha_2+72\alpha_0+25\alpha_1+49\alpha_2+48)i \\ %& i=0,1,\ldots\\ 
				& + 18\alpha_0\alpha_2+8\alpha_2\alpha_0^2
        +4\alpha_1+4\alpha_2^2\alpha_1+8\alpha_1\alpha_2+2\alpha_0^3+5\alpha_2^2\alpha_0+8\alpha_2\alpha_1\alpha_0 \\
		& +12\alpha_2+7+15\alpha_0	+\alpha_2^2\alpha_1^2+10\alpha_0^2+6\alpha_0\alpha_1
			+2\alpha_2\alpha_1^2
		+2\alpha_0^2\alpha_1+\alpha_1^2\alpha_0\\ &+\Bigl.5\alpha_2^2 +\alpha_2\alpha_0^3+\alpha_2^2\alpha_0^2+\alpha_1^2+\alpha_2\alpha_1^2\alpha_0+2\alpha_2\alpha_0^2\alpha_1+2\alpha_2^2\alpha_0\alpha_1\Bigr) \\
		&\times \Bigl((3i+\alpha_0+\alpha_2+1)(3i+\alpha_0+\alpha_1+2)(3i+\alpha_0+\alpha_2+3)(3i+\alpha_0+\alpha_1+3)\Bigr)^{-1};\\
%		(3i+\alpha_0+\alpha_2+1)^{-1}(3i+\alpha_0+\alpha_1+2)^{-1}(3i+\alpha_0+\alpha_2+3)^{-1}(3i+\alpha_0+\alpha_1+3)^{-1};\\
		\end{array}
$$
%$$ \begin{array}{@{}l@{\hspace{.4mm}}ll}
$$ \begin{array}{@{}l@{\hspace{.4mm}}l@{\hspace{-1.4mm}}l} 
 %%%%%%%%%%%%%%%%%%%%%%%%%%%%%%%%%%%%%%%%%%%%%%%%%%
c_1&=(1+\alpha_0)(1+\alpha_1)(3+\alpha_0+\alpha_1)^{-1}(2+\alpha_0+\alpha_1)^{-2};\\
\mbox{\tt for}\;&i=1,2,\ldots\\
%c_1&= (1+\alpha_0)(1+\alpha_1)(3+\alpha_0+\alpha_1)^{-1}(2+\alpha_0+\alpha_1)^{-2}\\
	 c_{2i}&=i(2i+\alpha_0)(2i+\alpha_0+\alpha_1)(2i+\alpha_0+\alpha_2)\Bigl(54i^4+(63\alpha_0+45\alpha_1+45\alpha_2)i^3  \Bigr.\\
   & +(24\alpha_0^2+8\alpha_1^2+8\alpha_2^2+42\alpha_0\alpha_1 +42\alpha_0\alpha_2+44\alpha_1\alpha_2-8)i^2 +(3\alpha_0^3+\alpha_1^3\\ 
		&+\alpha_2^3+12\alpha_0^2\alpha_1+12\alpha_0^2\alpha_2+3\alpha_0\alpha_1^2+3\alpha_0\alpha_2^2+33\alpha_0\alpha_1\alpha_2  +8\alpha_1^2\alpha_2+8\alpha_1\alpha_2^2 \\ %& i=1,2,\ldots\\
		&-3\alpha_0-4\alpha_1-4\alpha_2)i+\alpha_0^3\alpha_1+\alpha_0^3\alpha_2+6\alpha_0^2\alpha_1\alpha_2+\alpha_1^3\alpha_2+\alpha_1\alpha_2^3+3\alpha_0\alpha_1^2\alpha_2 \\ &+3\alpha_0\alpha_1\alpha_2^2-\alpha_0\alpha_1-\Bigl.\alpha_0\alpha_2-2\alpha_1\alpha_2\Bigr)\Bigl((3i+\alpha_0+\alpha_1+1)(3i+\alpha_0+\alpha_2+1)\Bigr)^{-1}\\
		&\Bigl((3i+\alpha_0+\alpha_1)^{2}(3i+\alpha_0+\alpha_2)^{2}(3i+\alpha_0+\alpha_1-1)(3i+\alpha_0+\alpha_2-1)\Bigr)^{-1};\\
  %%%%%%%%%%%%%%%%%%%%%%%%%%%%%%%%%%%%%%%%%%
	%\end{array}
%$$
%$$ 
%\begin{array}{@{}l@{\hspace{.4mm}}ll}
%c_1&=(1+\alpha_0)(1+\alpha_1)(3+\alpha_0+\alpha_1)^{-1}(2+\alpha_0+\alpha_1)^{-2};\\
%\mbox{\tt for}&i=1,2,\ldots\\
  c_{2i+1}&=(2i+\alpha_0+1)(2i+\alpha_0+\alpha_1+1)(2i+\alpha_0+\alpha_2+1)\Bigl(54i^5+(63\alpha_0+45\alpha_1 \Bigr.\\ &+45\alpha_2+135)i^4 
     +(24\alpha_0^2+8\alpha_1^2  +8\alpha_2^2+42\alpha_0\alpha_1+42\alpha_0\alpha_2+44\alpha_1\alpha_2+126\alpha_0 \\
		& +76\alpha_1+104\alpha_2+120)i^3+(3\alpha_0^3+\alpha_1^3+\alpha_2^3+12\alpha_0^2\alpha_1+12\alpha_0^2\alpha_2+3\alpha_0\alpha_1^2 \\ &+3\alpha_0\alpha_2^2+33\alpha_0\alpha_1\alpha_2
     +8\alpha_1^2\alpha_2+8\alpha_1\alpha_2^2+36\alpha_0^2+5\alpha_1^2+19\alpha_2^2+54\alpha_0\alpha_1 \\ % &i=1,2,\ldots\\
		&+72\alpha_0\alpha_2+66\alpha_1\alpha_2+87\alpha_0 +39\alpha_1+81\alpha_2+45)i^2
  +(\alpha_0^3\alpha_1+\alpha_0^3\alpha_2 \\
		&+6\alpha_0^2\alpha_1\alpha_2+\alpha_1^3\alpha_2+\alpha_1\alpha_2^3+3\alpha_0\alpha_1^2\alpha_2+3\alpha_0\alpha_1\alpha_2^2+3\alpha_0^3+2\alpha_2^3 + 12\alpha_0^2\alpha_1\\ 
	&+12\alpha_0^2\alpha_2+6\alpha_0\alpha_2^2+33\alpha_0\alpha_1\alpha_2+5\alpha_1^2\alpha_2+11\alpha_1\alpha_2^2+18\alpha_0^2+20\alpha_0\alpha_1\\ 
	&+38\alpha_0\alpha_2+14\alpha_2^2+26\alpha_1\alpha_2 
	 +24\alpha_0+6\alpha_1+24\alpha_2+6)i + 
  \alpha_0^3\alpha_1+3\alpha_0^2\alpha_1\alpha_2\\
	&+3\alpha_0\alpha_1\alpha_2^2+\alpha_1\alpha_2^3+\alpha_0^3+\alpha_2^3+3\alpha_0^2\alpha_1+3\alpha_0^2\alpha_2+6\alpha_0\alpha_1\alpha_2+3\alpha_0\alpha_2^2\\ 
	& +3\alpha_1\alpha_2^2+3\alpha_0^2+3  \alpha_2^2+2\alpha_0\alpha_1+6\alpha_0\alpha_2+2\alpha_1\alpha_2+2\alpha_0+2\alpha_2\Bigr) \\ &\times(3i+\alpha_0+\alpha_1+3)^{-1}(3i+\alpha_0+\alpha_2+2)^{-1}(3i+\alpha_0+\alpha_1+2)^{-2}\\
	&
  \times(3i+\alpha_0+\alpha_2+1)^{-2}(3i+\alpha_0+\alpha_1+1)^{-1}(3i+\alpha_0+\alpha_2)^{-1};\\
		\end{array}
$$
$$ \begin{array}{@{}l@{\hspace{.4mm}}ll} 
%$$ \begin{array}{lll} 
%%%%%%%%%%%%%%%%%%%%%%%%%%%%%%%%%%%%%%%%%   
\mbox{\tt for}\;&i=1,2,\ldots\\
d_{2i}&=i(2i+\alpha_0)(2i+\alpha_0-1)(2i+\alpha_0+\alpha_1)(2i+\alpha_0+\alpha_1-1)(2i+\alpha_0+\alpha_2)\\
& \times(2i+\alpha_0+\alpha_2-1)(i+\alpha_1) (i+\alpha_1-\alpha_2) (3i+\alpha_0+\alpha_1+1)^{-1}\\
    &\times(3i+\alpha_0+\alpha_1)^{-2}(3i+\alpha_0+\alpha_2)^{-1}(3i-1+\alpha_0+\alpha_1)^{-2}\\
	 &\times(3i+\alpha_0+\alpha_2-1)^{-1}(3i+\alpha_0+\alpha_1-2)^{-1}(3i+\alpha_0+\alpha_2-2)^{-1},  \\
	%%%%%%%%%%%%%%%%%%%%%%%%%%%%%%%%%%%%%%%%%%%%%%%%%%%%%%%  
     d_{2i+1}&=i(2i+\alpha_0+1)(2i+\alpha_0)(2i+\alpha_0+\alpha_1)(2i+\alpha_0+\alpha_1+1)\\
    &\times(2i+\alpha_0+\alpha_2+1)(2i+\alpha_0+\alpha_2)(i+\alpha_2(i+\alpha_2-\alpha_1)\\
		 &\times(3i+\alpha_0+\alpha_1+2)^{-1}(3i+\alpha_0+\alpha_2+2)^{-1}\times(3i+\alpha_0+\alpha_1+1)^{-1}\\
		 &\times(3i+1+\alpha_0+\alpha_2)^{-2}(3i+\alpha_0+\alpha_1)^{-1}(3i+\alpha_0+\alpha_2)^{-2}(3i+\alpha_0+\alpha_2-1)^{-1}.
\end{array}
$$
\noindent
Coefficients (\ref{eq:D}), involved in the computation of the weights:
$$
\begin{array}{ll}
f_{1,1}= \frac{ \Gamma(1 + \alpha_0) \Gamma(1 + \alpha_1)}{\Gamma(2 + \alpha_0+\alpha_1)}, & \\
f_{2,1}= \frac{\Gamma(1 + \alpha_0) \Gamma(1 + \alpha_2)}{\Gamma(2 + \alpha_0+\alpha_2)}, &
f_{2,2}= \left((1+\alpha_2) -(2+\alpha_0+\alpha_2)b_0\right)\frac{\Gamma(1 + \alpha_0) \Gamma(1 + \alpha_2)}{\Gamma(3 + \alpha_0+\alpha_2)}, \\
\end{array}
$$
where
$\Gamma $ is the Gamma function  \cite[p.~255]{abr64} defined as
$$
\Gamma(z)= \int_{0}^{\infty} t^{z-1} e^{-t} dt, \qquad  \mathcal{R} (z) >0,
$$
and computed by the  {\tt Matlab} function {\tt gamma.m}. 

%\noindent
%Input parameters of {\tt GaussMOP.m}:
%$$
%\begin{array}{l}
  %{\tt IC}= 1\\
 %{ n}: \mbox{ number of nodes of the simultaneous Gaussian quadrature rule}\\
  %\boldsymbol{\alpha} = [\alpha_0,\;\alpha_1,\; \alpha_2]^T, $  \;\; $\alpha_j >-1,\; j=0,1,2,  \;\; \alpha_1-\alpha_2 \notin \mathbb{Z}  \\
%\end{array}
%$$
\noindent
References: \cite{Pine,VanAssche2001}.

%%%%%%%%%%%%%%%%%%%%%%%%%%%%%%%%%%%%%%%%%%%%%%%%%%%%%%%%%%%%%%%%%%%%%%%%%%%%%%%%%%%%%%%%%%%%%%%%
\bigskip

\noindent
$\mbox{\bf MOP}_2:$  %\subSection
{Multiple Laguerre polynomials of first kind} % \label{example3}
\\ %$ $\\
\noindent
Weights: %$ x^{\alpha_j}e^{-x},\; \alpha_j >-1, \; j=1,\ldots, r.$\\
$$\left(  w^{(1)}(x), w^{(2)}(x) \right):= \left(x^{\alpha_1}e^{-x}, x^{\alpha_2}e^{-x}\right), $$
$ \alpha_j >-1,\; j=1,2. $\\

\noindent
Interval: $\Delta^{(i)}=[0, \infty),\; i=1,2.$ \\

\noindent
Recurrence relation coefficients:
$$ \begin{array}{l@{\hspace{.4mm}}l} \\
\mbox{\tt for}\;&i=0,1,2,\ldots\\
&b_{2i}=3i+\alpha_1+1, \\
& b_{2i+1}=3i+\alpha_2+2,\\
&c_{2i}=i(3i+\alpha_1+\alpha_2),\\
& c_{2i+1}=3i^2+(\alpha_1+\alpha_2+3)i+\alpha_1+1,\\
&d_{2i}=i(i+\alpha_1)(i+\alpha_1-\alpha_2),\\
& d_{2i+1}=i(i+\alpha_2)(i+\alpha_2-\alpha_1).\\
\end{array}
$$
%%%%%%%%%%%
\noindent
Coefficients (\ref{eq:D}), involved in the computation of the weights:
$$
\begin{array}{ll}
f_{1,1}=\Gamma(1 + \alpha_1),& \\
f_{2,1}= \Gamma(1 + \alpha_2),  &
f_{2,2}= \Gamma(1 + \alpha_2) ( \alpha_2-\alpha_1).
% f_{2,2}= (1 + \alpha_2)\Gamma(1 + \alpha_2)- b_0 \Gamma(1 + \alpha_2) % ( b0= 1+\alpha_1 )
\end{array}
$$

%\noindent
%Input parameters of {\tt GaussMOP.m}:
%$$
%\begin{array}{l}
  %{\tt IC}= 2\\
 %{ n}: \mbox{ number of nodes of the simultaneous Gaussian quadrature rule}\\
  %\boldsymbol{\alpha} = [\alpha_1,\; \alpha_2]^T, $ \;\; $\alpha_j >-1,\; j=1,2
%\end{array}
%$$
\noindent
References: \cite{Soro1,Soro3,VanAssche2001}.
%%%%%%%%%%%%%%%%%%%%%%%%%%%%%%%%%%%%%%%%%%%%%%%%%%%%%%%%%%%%%%%%%%%%%%%%%%%%%%%%%%%%%%%%%%%%%%%%
\bigskip

\noindent
$\mbox{\bf MOP}_3:$  %\subSection
{Multiple Laguerre polynomials of second kind}% \label{example3.1}
\\ %$ $\\
\noindent
Weights: %$ x^{\alpha_j}e^{-x},\; \alpha_j >-1, \; j=1,\ldots, r.$\\
$$\left(  w^{(1)}(x), w^{(2)}(x) \right):= \left(x^{\alpha_0}e^{-\alpha_1 x}, x^{\alpha_0}e^{-\alpha_2 x}\right), $$ $ \alpha_0 >-1,\;\alpha_j >0, \; j=1,2,\; \alpha_1 \ne \alpha_2.$ \\

\noindent
Interval: $\Delta^{(i)}=[0, \infty),\; i=1,2.$ \\

\noindent
Recurrence relation coefficients:
%$$ \begin{array}{l} \\
$$ \begin{array}{l@{\hspace{.4mm}}l} \\
\mbox{\tt for}\;&i=0,1,2,\ldots\\
&b_{2i}=\frac{\displaystyle i(\alpha_1+3\alpha_2) +(1+\alpha_0)\alpha_2}{\displaystyle \alpha_1\alpha_2},\\
& b_{2i+1}=\frac{\displaystyle i(3\alpha_1+\alpha_2) +(2+\alpha_0)\alpha_1+\alpha_2}{\displaystyle \alpha_1\alpha_2},\\
&c_{2i}=\frac{\displaystyle i(2i+\alpha_0)(\alpha_1^2+\alpha_2^2)}{\displaystyle \alpha_1^2\alpha_2^2},\\
& c_{2i+1}=\frac{\displaystyle 2i^2(\alpha_1^2+\alpha_2^2)+i\left(\alpha_1^2+3\alpha_2^2+\alpha_0(\alpha_1^2+\alpha_2^2)\right)+(1+\alpha_0)\alpha_2^2}{\displaystyle \alpha_1^2\alpha_2^2},\\
&d_{2i}=\frac{\displaystyle i(2i+\alpha_0)(2i+\alpha_0-1)(\alpha_2-\alpha_1)}{\displaystyle \alpha_1^3\alpha_2},\\
& d_{2i+1}=\frac{\displaystyle i(2i+\alpha_0)(2i+\alpha_0+1)(\alpha_1-\alpha_2)}{\displaystyle \alpha_1\alpha_2^3}.\\
\end{array}
$$
%%%%%%%%%%%
\noindent
Coefficients (\ref{eq:D}), involved in the computation of the weights:
$$
\begin{array}{ll}
f_{1,1}=\alpha_1^{-1-\alpha_0}\Gamma(1 + \alpha_0),& \\
f_{2,1}= \alpha_2^{-1-\alpha_0}\Gamma(1 + \alpha_0),  &
f_{2,2}= \alpha_1^{-1}\alpha_2^{-2-\alpha0}  (\alpha_1-\alpha_2)\Gamma(2 + \alpha_0).
% f_{2,2}= (1 + \alpha_2)\Gamma(1 + \alpha_2)- b_0 \Gamma(1 + \alpha_2) % ( b0= 1+\alpha_1 )
\end{array}
$$
%\noindent
%Input parameters of {\tt GaussMOP.m}:
%$$
%\begin{array}{l}
  %{\tt IC}= 3\\
 %{ n}: \mbox{ number of nodes of the simultaneous Gaussian quadrature rule}\\
  %\boldsymbol{\alpha} = [\alpha_0,\;\alpha_1,\; \alpha_2]^T,  \;\; \alpha_0 >-1,\;\alpha_j >0, \; j=1,2,\; \alpha_1 \ne \alpha_2. \\
%\end{array}
%$$
\noindent
References:\cite{Niki1,VanAssche2001}.
%\end{example}
%\begin{example}
%%%%%%%%%%%%%%%%%%%%%%%%%%%%%%%%%%%%%%%%%%%%%%%%%%%%%%%%%%%%%%%%%%%%%%%%%%%%%%%%%%%%%%%%%%%%%%%%
\bigskip

\noindent
$\mbox{\bf MOP}_4:$  %\subSection
{Multiple Hermite polynomials}% \label{example4.1}
\\%$ $\\
\noindent
Weights:% $ x^{c_j x}e^{-x^2},\; c_j\in \RR, \; j=1,\ldots, r.$

$$\left(  w^{(1)}(x), w^{(2)}(x) \right):= \left(e^{-x^2+\alpha_1 x}, e^{-x^2+\alpha_2 x}\right), $$ 
$\alpha_1, \alpha_2\in \RR,  \;\; \alpha_1 \ne \alpha_2.$\\

\noindent
Intervals: $\Delta^{(i)}=(-\infty, \infty), \; i=1,2.$ \\
%Interval: $(-\infty, \infty).$ \\

\noindent
Recurrence relation coefficients:
%$$ \begin{array}{l} \\
$$ \begin{array}{l@{\hspace{.4mm}}l} \\
\mbox{\tt for}\;&i=0,1,2,\ldots\\
&b_{2i}=\alpha_1/2,\\
& b_{2i+1}=\alpha_2/2,\\
&c_{i}=i/2, \\ 
&d_{2i}=i(\alpha_1-\alpha_2)/4,\\ 
&d_{2i+1}=i(\alpha_2-\alpha_1)/4.\\
\end{array}
$$
\noindent
Coefficients (\ref{eq:D}), involved in the computation of the weights:
%%%%%%%%%%%
$$
\begin{array}{ll}
f_{1,1}=e^{\frac{\alpha_1^2}{4}}  \sqrt{\pi},& \\
f_{2,1}=  e^{\frac{\alpha_2^2}{4}}  \sqrt{\pi},  &
f_{2,2}= \frac{\alpha_2-\alpha_1}{2} e^{\frac{\alpha_2^2}{4}}  \sqrt{\pi}.
% f_{2,2}= (1 + \alpha_2)\Gamma(1 + \alpha_2)- b_0 \Gamma(1 + \alpha_2) % ( b0= 1+\alpha_1 )
\end{array}
$$
%\noindent
%Input parameters of {\tt GaussMOP.m}:
%$$
%\begin{array}{l}
  %{\tt IC}= 4\\
 %{ n}: \mbox{ number of nodes of the simultaneous Gaussian quadrature rule}\\
  %\boldsymbol{\alpha} = [\alpha_1,\; \alpha_2]^T, \;\;\alpha_1, \alpha_2\in \RR,  \;\; \alpha_1 \ne \alpha_2 \\
%\end{array}
%$$
\noindent
References: \cite{VanAssche2001}.
\bigskip

\noindent
$\mbox{\bf MOP}_5:$    % Verified 
{Multiple Laguerre--Hermite polynomials}%\label{example6.1}
\\%$ $\\
\noindent
Weights:% $ x^{c_j x}e^{-x^2},\; c_j\in \RR, \; j=1,\ldots, r.$

$$\left(  w^{(1)}(x), w^{(2)}(x) \right):= \left(e^{-x^2} |x|^{\beta}, e^{-x^2} x^{\beta}\right),$$
$ \beta >-1.$ \\

\noindent
Intervals: $\Delta^{(1)}=(-\infty, 0],\;\Delta^{(2)}=[0,\infty).$ \\

\noindent
Recurrence relation coefficients:
$$ \begin{array}{l@{\hspace{.4mm}}l} \\
\mbox{\tt for}\;&i=0,1,2,\ldots\\
%$$ \begin{array}{l} \\
&b_{2i}=X_{i}^{(\beta)},\\
& b_{2i+1}=-X_{i}^{(\beta)},\\
&c_{2i}=\frac{i}{2}, \\
& c_{2i+1}=\frac{2i+\beta+1}{2}-(X_i^{(\beta)})^2\\
&d_{2i}=\frac{i}{2}X_{i-1}^{(\beta)},\\
& d_{2i+1}=-\frac{i}{2}X_{i}^{(\beta)}
\end{array}
$$
with $
X_{i}^{(\beta)}=-\frac{\Gamma\left(\frac{i+\beta+2}{2}\right)}{\Gamma\left(\frac{i+\beta+1}{2}\right)}.
$
Moreover,  for large $ i, $ $X_{i}^{(\beta)}= -\sqrt{\frac{\beta+i}{2}}+ o (\sqrt{i}) $ \cite{VanAssche2001}.
%%%%%%%%%%%

\noindent
Coefficients (\ref{eq:D}), involved in the computation of the weights:
$$
\begin{array}{ll}
f_{1,1}=%\int_{-\infty}^{0}e^{-x^2} |x|^{\beta}dx=
  \frac{1}{2} \Gamma\left(\frac{1+\beta}{2}\right),& \\
f_{2,1}=%\int_{0}^{\infty}e^{-x^2} x^{\beta}dx= 
 \frac{1}{2} \Gamma\left(\frac{1+\beta}{2}\right),  &
f_{2,2}=%\int_{0}^{\infty}e^{-x^2} x^{\beta}(x-b_0)p_0(x)dx=
\frac{1}{2}\left(-b_0 \Gamma\left(\frac{1+\beta}{2}\right)+\Gamma\left(\frac{2+\beta}{2}\right)\right).
%\frac{1-b_0}{2} \Gamma\left(\frac{1+\beta}{2}\right).
% f_{2,2}= (1 + \alpha_2)\Gamma(1 + \alpha_2)- b_0 \Gamma(1 + \alpha_2) % ( b0= 1+\alpha_1 )
\end{array}
$$
%\noindent
%Input parameters of {\tt GaussMOP.m}:
%$$
%\begin{array}{l}
  %{\tt IC}= 5\\
 %{ n}: \mbox{ number of nodes of the simultaneous Gaussian quadrature rule}\\
  %\boldsymbol{\alpha} = [\beta]^T,   \;\; \beta >-1  \\
%\end{array}
%$$
\noindent
References: \cite{Soro3,VanAssche2001}.
%%%%%%%%%%%%%%%%%%%%%%%%%%%%%%%%%%%%%%%%%%%%%%%%%%%%%%%%%%%%%%%%%%%%%%%%%%%%%%%%%%%%%%%%%%%%%%%%
\bigskip

\noindent
$\mbox{\bf MOP}_6:$  % Verified
{MOPs Associated with
the Modified Bessel function of the  second kind (Macdonald function) $K_\nu(x)$}\label{example1}
%\begin{example}\label{example1}
\\ %$ $\\
\noindent
Weights:
$$\left(  w^{(1)}(x), w^{(2)}(x) \right):= \left(2x^{\alpha +\nu/2}K_\nu(2\sqrt{x}),2x^{\alpha +(\nu+1)/2}K_{\nu+1}(2\sqrt{x}) \right),$$ 
$ \alpha >-1,\;\; \nu\ge 0,$ 
with $K_\nu(x)$ % and $K_{\nu+1}(x)$
the modified Bessel function  of the second kind given by
$$   K_\nu(x):= \frac12\left(\frac{x}{2}\right)^\nu\int_{0}^\infty exp\left(-t-\frac{x^2}{4t}\right)t^{-\nu-1}dt.
$$

\noindent
Interval: $\Delta^{(i)}=[0, \infty), i=1,2.$ \\

\noindent
Recurrence relation coefficients:
%and for the weight functions 
%it is shown in \cite{VanAssche} that, for $\alpha >-1$ and $\nu\ge 0$, the recurrence relations for the corresponding monic MOPs are given by $a_i=1$ for $i\ge 0$, and 
%\begin{eqnarray} \nonumber
%b_i & \! = \! & i(3i+\alpha+2\nu) + (\alpha+1)(3i+\alpha +\nu+1), \quad  i\ge 0  \\ \nonumber
%c_i & \! = \! & i(i+\alpha)(i+\alpha+\nu)(3i+2\alpha+\nu), \quad i\ge 1 \\ \nonumber
%d_i & \! = \!  & i(i-1)(i+\alpha)(i+\alpha-1)(i+\alpha+\nu)(i+\alpha+\nu-1), \quad i\ge 2 .
%\end{eqnarray}
$$ \begin{array}{l@{\hspace{.4mm}}l} \\
\mbox{\tt for}\;&i=0,1,2,\ldots\\
& b_i  =  i(3i+\alpha+2\nu) + (\alpha+1)(3i+\alpha +\nu+1),  \\ 
& c_i  =  i(i+\alpha)(i+\alpha+\nu)(3i+2\alpha+\nu), \\
& d_i  =  i(i-1)(i+\alpha)(i+\alpha-1)(i+\alpha+\nu)(i+\alpha+\nu-1).
\end{array}
$$

\noindent
Coefficients (\ref{eq:D}), involved in the computation of the weights:
$$
\begin{array}{ll}
f_{1,1}= \Gamma(\alpha+1)\Gamma(\alpha+\nu+1),& \\
f_{2,1}=  \Gamma(\alpha+1)\Gamma(\alpha+\nu+2),  &
f_{2,2}=\Gamma(\alpha+2)\Gamma(\alpha+\nu+2).\\ %(\alpha+1)^2(\nu+1)\Gamma(\alpha+1)\Gamma(\alpha+\nu+2).\\
%(\alpha+1)(\alpha+\nu +\alpha \nu+1)\Gamma(\alpha+1)\Gamma(\alpha+\nu+2).
\end{array}
$$
%
%
%\noindent
%Input parameters of {\tt GaussMOP.m}:
%$$
%\begin{array}{l}
  %{\tt IC}= 6\\
 %{ n}: \mbox{ number of nodes of the simultaneous Gaussian quadrature rule}\\
  %\boldsymbol{\alpha} = [\alpha,\;  \nu]^T,   \;\; \alpha >-1,\;\; \nu\ge 0 \\
%\end{array}
%$$
\noindent
References: \cite{ChDo2000,VanAssche2000,VanAssche2023}.
%%%%%%%%%%%%%%%%%%%%%%%%%%%%%%%%%%%%%%%%%%%%%%%%%%%%%%%%%%%%%%%%%%%%%%%%%%%%%%%%%%%%%%%%%%%%%%%%
\bigskip

\noindent
$\mbox{\bf MOP}_7:$  % Verified
{MOPs associated with
the Modified Bessel functions of the first kind $I_\nu(x)$}\label{example8.1}
\\ %$ $\\
Weights:
$$\left( w^{(1)}(x), w^{(2)}(x) \right):= \left(x^{\nu/2 }I_\nu(2\sqrt{x})e^{-\beta x},x^{(\nu+1)/2}I_{\nu+1}(2\sqrt{x})e^{-\beta x} \right),$$ 
$\beta>0,\;\nu > -1$
where 
$I_\nu(x)$ is  the modified Bessel function  of the first kind \cite{VanAssche2003},
$$I_\nu(x)=\left(\frac{z}{2}\right)^\nu \sum_{k=0}^{\infty} \frac{\left(\frac{z}{2}\right)^{2k}}{k!\Gamma(\nu+k+1)},\quad \nu\in \RR.$$
\noindent
Interval: $\Delta^{(i)}=[0, \infty), i=1,2.$ \\

\noindent
Recurrence relation coefficients: 
%it is shown in \cite{VanAssche} that, for, the recurrence relations for the corresponding monic MOPs are given by $a_i=1$ for $i\ge 0$, and
%\begin{eqnarray} \nonumber
%b_i & \! = \! & \frac{1}{\beta^2}(1+\beta(\nu +2i+1) ), \quad  i\ge 0  \\ \nonumber
%c_i & \! = \! &  \frac{i}{\beta^3}(2+\beta(\nu+i)), \quad i\ge 1 \\ \nonumber
%d_i & \! = \!  & \frac{i(i -1)}{\beta^4}, \quad i\ge 2.
%\end{eqnarray}
$$ \begin{array}{l@{\hspace{.4mm}}l} \\
\mbox{\tt for}\;&i=0,1,2,\ldots\\
& b_i =  (1+\beta(\nu +2i+1) )/\beta^{2}, \\
& c_i =  i(2+\beta(\nu+i))/\beta^{3}, \\
& d_i = {i(i -1)}/{\beta^4}.
\end{array}
$$
\noindent
Coefficients (\ref{eq:D}), involved in the computation of the weights:
$$
\begin{array}{ll}
f_{1,1}= \beta^{-1-\nu} e^{\frac{1}{\beta}},& \\
f_{2,1}=  \beta^{-2-\nu} e^{\frac{1}{\beta}},&
f_{2,2}=\beta^{-3-\nu} e^{\frac{1}{\beta}}.\\
%(\alpha+1)(\alpha+\nu +\alpha \nu+1)\Gamma(\alpha+1)\Gamma(\alpha+\nu+2).
\end{array}
$$
%\noindent
%Input parameters of {\tt GaussMOP.m}:
%$$
%\begin{array}{l}
  %{\tt IC}= 7\\
 %{ n}: \mbox{ number of nodes of the simultaneous Gaussian quadrature rule}\\
  %\boldsymbol{\alpha} = [\beta,\;  \nu]^T,   \;\; \beta >0,\;\; \nu\ge -1 \\
%\end{array}
%$$
\noindent
References: \cite{VanAssche2003,VanAssche2023}.
%%%%%%%%%%%%%%%%%%%%%%%%%%%%%%%%%
%%%%%%%%%%%%%%%%%%%%%%%%%%%%%%%%%%%%%%%%%%%%%%%%%%%%%%%%%%%%%%%%%%%%%%%%%%%%%%%%%%%%%%%%%%%%%%%%
\bigskip

\noindent
$\mbox{\bf MOP}_8:$ % Verified
{MOPs associated with the
Gauss' hypergeometric function}\label{example9.1}
\\ %$ $\\
Weights:
%$$\left( w^{(1)}(x), w^{(2)}(x) \right):= 
%\left(\frac{\Gamma(c)\Gamma(d)}{\Gamma(a) \Gamma(b)\Gamma(\delta)}x^{a-1} (1-x)^{\delta-1}\;  _2{F}_1 \left( \left[\begin{array}{@{}c@{}}c-b\\ d-b \end{array}\right],\delta,1-x \right),\frac{\Gamma(c)\Gamma(d)}{\Gamma(a) \Gamma(b)\Gamma(\delta)}x^{a-1} (1-x)^{\delta-1}\;  _2{F}_1 \left( \left[\begin{array}{@{}c@{}}c-b\\ d-b \end{array}\right],\delta,1-x \right) \right),
$$ 
\begin{array}{@{}l}
 w^{(1)}(x)  :=  
\frac{\Gamma(c)\Gamma(d)}{\Gamma(a) \Gamma(b)\Gamma(\delta)}x^{a-1} (1-x)^{\delta-1}\;  _2{F}_1 \left( \left[\begin{array}{@{}c@{}}c-b\\ d-b \end{array}\right],\delta,1-x\right), \\
 w^{(2)}(x) := \frac{\Gamma(c+1)\Gamma(d)}{\Gamma(a) \Gamma(b+1)\Gamma(\delta)}x^{a-1} (1-x)^{\delta-1}\;  _2{F}_1 \left( \left[\begin{array}{@{}c@{}}c-b\\ d-b-1 \end{array}\right],\delta,1-x \right),
\end{array} 
$$
with
$
 a,b,c,d \in \RR^{+},  \; c+1, d > a, \; c,d > b, \; \delta= c+d-b-a>0,$
 and
$$
 _2{F}_1 \left( \left[\begin{array}{@{}c@{}}\alpha\\ \beta \end{array}\right],\gamma,z\right)=
\sum_{i=0}^{\infty} \frac{(\alpha)_i (\beta)_i}{(\gamma)_i} \frac{z^n}{n!}
$$
is the Gauss' hypergeometric function,  computed by the {\tt Matlab} function {\tt hypergeom.m}, and $ (\mu)_i $ denotes the Pochhammer symbol defined by
\begin{eqnarray*}
(\mu)_0 &  :=& 0, \\
(\mu)_i &  :=&\mu(\mu+1) \cdots(\mu+i-1),\;\; i \in \ZZ^{+}. 
\end{eqnarray*}
%where 
%$I_\nu(x)$ is  the modified Bessel function  of the first kind \cite{VanAssche2003},
%$$I_\nu(x)=\left(\frac{z}{2}\right)^\nu \sum_{k=0}^{\infty} \frac{\left(\frac{z}{2}\right)^{2k}}{k!\Gamma(\nu+k+1)},\quad \nu\in \RR.$$
\noindent
Interval: $\Delta^{(i)}=[0, 1], i=1,2.$ \\

\noindent
Recurrence relation coefficients: \\
Let 
\begin{eqnarray*}
\lambda_{3i} & = &  \frac{i (a+i-1)(c'_i -b -1)}{(c'_i +i-2)(c'_i +i-1)(c'_{i+1}+i-2)},\\
\lambda_{3i+1} & = &  \frac{i (b+i)(c'_{i+1} -a -1)}{(c'_i +i-1)(c'_{i+1} +i-2)(c'_{i+1}+i-1)}, \\
\lambda_{3i+2} & = &  \frac{(a+i)(b+i)(c'_i- 1)}{(c'_{i}+i-1)(c'_{i} +i)(c'_{i+1}+i-1)},
\end{eqnarray*}
with
$$
c'_i=\left\{ \begin{array}{ll} c+k, & \;\; \mbox{\rm if} \;\;  i =2 k-1, 
\\
d+k, & \;\; \mbox{\rm if} \;\;  i =2 k.
\end{array} \right.
$$
Then
%it is shown in \cite{VanAssche} that, for, the recurrence relations for the corresponding monic MOPs are given by $a_i=1$ for $i\ge 0$, and 
%\begin{eqnarray*} 
%b_i & \! = \! & \lambda_{3i}\lambda_{3i+1}\lambda_{3i+2},\\
%c_i & \! = \! & \lambda_{3i+1}\lambda_{3i+3}+\lambda_{3i+2}\lambda_{3i+3}+\lambda_{3i+2}\lambda_{3i+4},\\
%d_i & \! = \!  & \lambda_{3i+2}\lambda_{3i+4}\lambda_{3i+6}.\\
%\end{eqnarray*}
$$ \begin{array}{l@{\hspace{.4mm}}l} \\
\mbox{\tt for}\;&i=0,1,2,\ldots\\
&b_i = \lambda_{3i}\lambda_{3i+1}\lambda_{3i+2},\\
&c_i = \lambda_{3i+1}\lambda_{3i+3}+\lambda_{3i+2}\lambda_{3i+3}+\lambda_{3i+2}\lambda_{3i+4},\\
&d_i = \lambda_{3i+2}\lambda_{3i+4}\lambda_{3i+6}.\\
\end{array}
$$
\noindent
Coefficients (\ref{eq:D}), involved in the computation of the weights:
$$
\begin{array}{ll}
f_{1,1}= 1,& \\
f_{2,1}=  1,&
f_{2,2}=\frac{a(c-b)}{cd (c+1)}.\\
%(\alpha+1)(\alpha+\nu +\alpha \nu+1)\Gamma(\alpha+1)\Gamma(\alpha+\nu+2).
\end{array}
$$
%
%\noindent
%Input parameters of {\tt GaussMOP.m}:
%$$
%\begin{array}{l}
  %{\tt IC}= 8\\
 %{ n}: \mbox{ number of nodes of the simultaneous Gaussian quadrature rule}\\
  %\boldsymbol{\alpha} = [a,\; b,\;c,\; d]^T,   \;\; a,b,c,d \in \RR^{+},  \;\; c+1, d > a, \;\;  c,d > b.
%\end{array}
%$$
\noindent
References: \cite{Lima22}.
%%%%%%%%%%%%%%%%%%%%%%%%%%%%%%%%%%%%%%%%%%%%%%%%%%%%%%%%%%%%%%%%%%%%%%%%%%%%%%%%%%%%%%%%%%%%%%%%
%%%%%%%%%%%%%%%%%%%%%%%%%%%%%%%%%%%%%%%%%%%%%%%%%%%%%%%%%%%%%%%%%%%%%%%%%%%%%%%%%%%%%%%%%%%%%%%
\bigskip

\noindent
$\mbox{\bf MOP}_9:$ % Verified
{MOPs associated with the
confluent hypergeometric function}\label{example10.1}
\\ %$ $\\
Weights:
%$$\left( w^{(1)}(x), w^{(2)}(x) \right):= 
%\left(\frac{\Gamma(c)\Gamma(d)}{\Gamma(a) \Gamma(b)\Gamma(\delta)}x^{a-1} (1-x)^{\delta-1}\;  _2{F}_1 \left( \left[\begin{array}{@{}c@{}}c-b\\ d-b \end{array}\right],\delta,1-x \right),\frac{\Gamma(c)\Gamma(d)}{\Gamma(a) \Gamma(b)\Gamma(\delta)}x^{a-1} (1-x)^{\delta-1}\;  _2{F}_1 \left( \left[\begin{array}{@{}c@{}}c-b\\ d-b \end{array}\right],\delta,1-x \right) \right),
%$$ 
\begin{eqnarray*}
 w^{(1)}(x) & := & 
\frac{\Gamma(c)}{\Gamma(a) \Gamma(b)}e^{-x}x^{a-1} U( c-b, a-b+1,x),  \\
 w^{(2)}(x)& :=& \frac{\Gamma(c+1)}{\Gamma(a) \Gamma(b)}e^{-x}x^{a-1} U( c-b+1, a-b+1,x),
\end{eqnarray*} 
where   $ U(a,b,z)  $ is the confluent hypergeometric function \cite[p.504--505]{abr64}, computed by the {\tt Matlab} function {\tt  kummerU.m},
$$
 U(a,b,z)=\frac{1}{\Gamma(a)} \int_{0}^{\infty}e^{-zt} t^{a-1}(1+t)^{b-a-1}  dt, \qquad  \mathcal{R} (z) >0,  \mathcal{R} (a) >0,
$$
  $ a,b,c \in \RR^{+}, \; c > \max\{a,b\}.  $ \\
	
%where 
%$I_\nu(x)$ is  the modified Bessel function  of the first kind \cite{VanAssche2003},
%$$I_\nu(x)=\left(\frac{z}{2}\right)^\nu \sum_{k=0}^{\infty} \frac{\left(\frac{z}{2}\right)^{2k}}{k!\Gamma(\nu+k+1)},\quad \nu\in \RR.$$
\noindent
Interval: $\Delta^{(i)}=[0, \infty), i=1,2.$ \\

\noindent
Recurrence relation coefficients: \\
$$ \begin{array}{@{}l@{\hspace{.4mm}}l@{}l} %l@{}l@{}l}
 \mbox{\tt for}\;&i=0,&1,2,\ldots\\
&b_{2i}&=\frac{\displaystyle (2 i+1) (a+2 i) (b+2 i)}{\displaystyle c+3 i}-\frac{\displaystyle 2 i (a+2 i-1) (b+2 i-1)}{\displaystyle c+3 i-1},\\
& b_{2i+1}&=\frac{\displaystyle (2 i+2) (a+2 i+1) (b+2 i+1)}{\displaystyle c+3 i+2}-\frac{\displaystyle (2 i+1) (a+2 i) (b+2 i)}{\displaystyle c+3 i},\\
&c_{2i}&=\frac{\displaystyle 2i(a+2 i-1) (b+2 i-1)}{\displaystyle c+3 i-1}
        \Bigl(\frac{\displaystyle (2 i-1) (a+2 i-2) (b+2 i-2)}{\displaystyle 2 (c+3 i-2)} \Bigr.\\
 &&      - \frac{\displaystyle 2 i (a+2 i-1) (b+2 i-1)}{\displaystyle c+3 i-1}
     \Bigl.+\frac{\displaystyle (2 i+1) (a+2 i) (b+2 i)}{\displaystyle 2 (c+3 i)}\Bigr), \\
 %c_{2i+1}&=&\frac{\displaystyle 2 i(2 i+1) (a+2 i) (b+2 i)  (a+2 i-1) (b+2 i-1)}{\displaystyle 2 (c+3 i) (c+3 i-1)}
        %-\frac{\displaystyle (2 i+2) (a+2 i+1) (b+2 i+1) (2 i+1) (a+2 i) (b+2 i)}{\displaystyle 2 (c+3 i+2) (c+3 i+1)}\\
     %&&  - \frac{\displaystyle (2 i+1) (a+2 i) (b+2 i)}{\displaystyle (c+3 i)^2}+\frac{\displaystyle (2 i+1) (a+2 i) (b+2 i)}{\displaystyle (c+3 i)} \frac{\displaystyle (2 i+2) (a+2 i+1) (b+2 i+1)}{\displaystyle c+3 i+2},\\
%%%%%%%%%%%%%%%%%%%%%%%%%%%
&c_{2i+1}&=\frac{\displaystyle (2 i+1) (a+2 i) (b+2 i)}{\displaystyle c+3 i}\Bigl( \frac{\displaystyle  i  (a+2 i-1) (b+2 i-1)}{\displaystyle   (c+3 i-1)} \Bigr.\\
  &&  \Bigl.   -\frac{\displaystyle (2 i+1) (a+2 i) (b+2 i)}{\displaystyle  (c+3 i) } %\right.\\
   %  && \left.
	+ \frac{\displaystyle ( i+1) (a+2 i+1) (b+2 i+1)}{\displaystyle c+3 i+1}\Bigr),\\		
&d_{2i}&=\frac{\displaystyle 2 i (2 i+1) (a+2 i-1) (a+2 i) (b+2 i-1) (b+2 i) (c+i-1) }
        {\displaystyle (c+3 i-2) (c+3 i-1) (c+3 i) (c+3 i-1) (c+3 i) (c+3 i+1)}\\ 
				&&\times (c-a+i) (c-b+i)\\
& d_{2i+1}&=\frac{\displaystyle (2 i+1) (2 i+2) (a+2 i) (a+2 i+1) (b+2 i) (b+2 i+1)}
        {\displaystyle (c+3 i) (c+3 i+1) (c+3 i+2)}.
\end{array}
$$
\noindent
Coefficients (\ref{eq:D}), involved in the computation of the weights:
$$
\begin{array}{ll}
f_{1,1}= 1,& \\
f_{2,1}=  1,&
f_{2,2}=-\frac{ab}{ c (c+1)}.\\
%(\alpha+1)(\alpha+\nu +\alpha \nu+1)\Gamma(\alpha+1)\Gamma(\alpha+\nu+2).
\end{array}
$$
%\noindent
%Input parameters of {\tt GaussMOP.m}:
%$$
%\begin{array}{l}
  %{\tt IC}= 9\\
 %{ n}: \mbox{ number of nodes of the simultaneous Gaussian quadrature rule}\\
  %\boldsymbol{\alpha} = [a,\; b,\;c]^T,   \;\;  a,b,c \in \RR^{+}, \; c > \max\{a,b\}.  
%\end{array}
%$$
\noindent
References: \cite{Lima20}.

In the sequel, we denote the $k$--th  MOP % evaluated in  $  \boldsymbol{\alpha}$
 by   $  \mbox{\bf MOP}_k, $ for $ k \in \{  1,2 , \ldots, 9\} $. 
%%%%%%%%%%%%%%%%%%%%%%%%%%%%%%%%%%%%%%%%%%%%%%%%%%%
\section{{\tt Matlab} function {\tt ClassMOP.m}}
\label{sect:class} 
Here,  the {\tt Matlab} function {\tt ClassMOP.m} % and its corresponding algorithm 
for computing the coefficients of the  recurrence relation associated with different classes of MOPs  is described.

The {\tt Matlab} command  is 
\begin{center}
{\tt [$\boldsymbol{b},\boldsymbol{ c},\boldsymbol{d}, F$]=ClassMOP(IC,$ n, \boldsymbol{\alpha} $)}
\end{center}
The input parameters of {\tt ClassMOP.m} are:
\begin{itemize}
\item  {\tt IC}: kind of MOPs.
\item  {$ n$}: number of nodes of the simultaneous Gaussian quadrature rule.
\item  {$ \boldsymbol{\alpha} $}: vector of parameters characterizing the weights.
\end{itemize}
The output parameters of {\tt ClassMOP} are:
\begin{itemize}
\item $\boldsymbol{b}$: the main diagonal of $ H_n.$
\item $\boldsymbol{ c}$: the first subdiagonal of $ H_n.$
\item $\boldsymbol{ d}$: the second subdiagonal of $ H_n.$
\item $F: $  the $ 2 \times 2 $ lower triangular matrix in (\ref{eq:D}).
\end{itemize}
The list of input parameters for the considered classes of MOPs is summarized in Table~\ref{tab:inputp}.
%The algorithm is summarized in the following steps: 
%\begin{enumerate}[label={\bf Step \arabic*}.]
%\item Choice of the MOP. \label{en:1a}
%%\item Construction of the Hessenberg matrix $ H_n $.\label{en:2a}
%\item Transformation of $ H_n$ to the similar matrix $ \hat{H}_n :=S_n^{-1} H_n S_n. $   \label{en:3a}
%\item Computation of the simultaneous Gaussian quadrature rule.\label{en:8a}
%\end{enumerate}
%All steps are described in detail in the following sections. 
\begin{table}
  \caption{Input parameters of the {\tt Matlab} function {\tt ClassMOP.m}}
  \label{tab:inputp}
	\centering
\begin{tabular}{|c|c|c|}%{|@{\hspace{.1cm}}l@{\hspace{.1cm}}|c|c|c|c|@{}c@{}|}
\hline
$\mbox{\bf MOP}$  & \# nodes & $ \boldsymbol{\alpha} $ \\ 
\hline
\hline
$1$ & $ n$ &$[\alpha_0,\;\alpha_1,\; \alpha_2]^T, $  \;\; $\alpha_j >-1,\; j=0,1,2,  \;\; \alpha_1-\alpha_2 \notin \mathbb{Z}   $ \\
\hline
$2$ & $ n$ &$  [\alpha_1,\; \alpha_2]^T, $ \;\; $\alpha_j >-1,\; j=1,2  $ \\
\hline
$3$ & $ n$ &$ [\alpha_0,\;\alpha_1,\; \alpha_2]^T,  \;\; \alpha_0 >-1,\;\alpha_j >0, \; j=1,2,\; \alpha_1 \ne \alpha_2.  $ \\
\hline
$4$ & $ n$ &$  [\alpha_1,\; \alpha_2]^T, \;\;\alpha_1, \alpha_2\in \RR,  \;\; \alpha_1 \ne \alpha_2  $ \\
\hline
$5$ & $ n$ &$ [\beta]^T,   \;\; \beta >-1    $ \\
\hline
$6$ & $ n$ &$  [\alpha,\;  \nu]^T,   \;\; \alpha >-1,\;\; \nu\ge 0  $ \\
\hline
$7$ & $ n$ &$  [\beta,\;  \nu]^T,   \;\; \beta >0,\;\; \nu\ge -1  $ \\
\hline
$8$ & $ n$ &$ [a,\; b,\;c,\; d]^T,   \;\; a,b,c,d \in \RR^{+},  \;\; c+1, d > a, \;\;  c,d > b.$ \\
\hline
$9$ & $ n$ &  $ [a,\; b,\;c]^T,   \;\;  a,b,c \in \RR^{+}, \; c > \max\{a,b\}.  $ 
 \\
\hline
\end{tabular}\end{table}

%%%%%%%%%%%%%%%%%%%%%%%%%%%%%%%%%%%%%%%%%%%%%%%%%%%%%%%%%%%%%%%%%%%%%%%%%%%%%%%%%%%%%%%%%%%%%%%%
%\end{enumerate}
%\newtheorem{remark}{Remark}
%\begin{remark} All the considered MOPs depend on the parameter vector $  \boldsymbol{\alpha}.$  In the sequel, we denote the $k$--th  MOP evaluated in  $  \boldsymbol{\alpha}$ by  $  \mbox{\bf MOP}k(\boldsymbol{\alpha}), $ for $ k \in \{  1,2 , \ldots, 9\} $ and  fixed $  \boldsymbol{\alpha}$.  
%\end{remark}
%\begin{remark} All the considered MOPs depend on the parameter vector $  \boldsymbol{\alpha}.$  
%In the sequel, we denote the $k$--th  MOP % evaluated in  $  \boldsymbol{\alpha}$
 %by   $  \mbox{\bf MOP}k, $ for $ k \in \{  1,2 , \ldots, 9\} $. %$  \mbox{\bf MOP}k(\boldsymbol{\alpha}), $ for $ k \in \{  1,2 , \ldots, 9\} $ and  fixed $  \boldsymbol{\alpha}$.  
%\end{remark}

%%%%%%%%%%%%%%%%%%%%%%%%%%%%%%%%%%%%%%%%%%%%%%%%%%%
\section{{\tt Matlab} function {\tt GaussMOP.m}}
\label{sect:alg} 
Here,  the {\tt Matlab} function {\tt GaussMOP.m} % and its corresponding algorithm 
for computing the simultaneous Gaussian quadrature rule associated with  different kinds of weights,  is reported.

The {\tt Matlab} command for computing the simultaneous Gaussian quadrature rule is 
\begin{center}
{\tt [$\boldsymbol{x},\boldsymbol{ \omega}^{(1)},\boldsymbol{ \omega}^{(2)}$,\tt ier]=GaussMOP($\boldsymbol{b},\boldsymbol{ c},\boldsymbol{d},n, F$)}
\end{center}
The input parameters of {\tt GaussMOP.m} are:
\begin{itemize}
\item $\boldsymbol{b}$: the main diagonal of $ H_n.$
\item $\boldsymbol{ c}$: the first subdiagonal of $ H_n.$
\item $\boldsymbol{ d}$: the second subdiagonal of $ H_n.$
\item  {$ n$}: number of nodes of the simultaneous Gaussian quadrature rule.
\item $F: $  the $ 2 \times 2 $ lower triangular matrix in (\ref{eq:D}).
\end{itemize}
The output parameters of {\tt GaussMOP.m} are:
\begin{itemize}
\item $\boldsymbol{x}$: vector of
 nodes of the simultaneous  Gauss quadrature rule.
\item $\boldsymbol{ \omega}^{(1)}$: vector of weights of the simultaneous Gaussian quadrature rule corresponding to $w^{(1)}(x).$ 
\item $\boldsymbol{ \omega}^{(2)}$: vector of weights of the simultaneous Gaussian quadrature rule corresponding to $w^{(2)}(x).$
\item {\tt ier}:  is set to zero for normal return, otherwise       {\tt ier} is set to $j$ if the $j$--th node has not been
          determined after 30 iterations.
\end{itemize}
The algorithm is summarized in the following steps, described in detail in Sections~\ref{subs:1} and~\ref{subs:2}, respectively: 
\begin{itemize}%\begin{enumerate}%[label={\bf Step \arabic*}.]
%\item Choice of the MOP. \label{en:1a}
%\item Construction of the Hessenberg matrix $ H_n $.\label{en:2a}
\item Transformation of $ H_n$ to the similar matrix $ \hat{H}_n :=S_n^{-1} H_n S_n. $   \label{en:3a}
\item Computation of the simultaneous Gaussian quadrature rule.\label{en:8a}
\end{itemize}%\end{enumerate}
%All steps are the following sections. 
%
%%%%%%%%%%%%%%%%%%%%%%%%%%%%%%%%%%%%%%%%%%%%%%%%%%%%%%%%%%%%%%
%\Section{\ref{en:2a} Construction of the Hessenberg matrix $ H_n $}
\section{Transformation of $ H_n$ to the similar matrix $ \hat{H}_n :=S_n^{-1} H_n S_n $}\label{subs:1}
As described in  Section~\ref{sect:intro}, the computation of simultaneous Gaussian quadrature rules associated with MOPs relies on that of the eigenvalues $ x_j $ and   corresponding left and right eigenvectors $ \boldsymbol{u}^{(j)} $  and $ \boldsymbol{v}^{(j)} $, $ j=1,\ldots,n, $  respectively, of $ H_n. $
As noticed in \cite{VanAssche2023}, the latter eigenvalue problem is very ill--conditioned for all classes of MOPs listed in  Section~\ref{sect:MOP}. Furthermore,  the {\tt Matlab} function {\tt balance.m}, commonly applied to reduce the  eigenvalue condition number, does not improve the condition of the considered eigenvalue problem,  and, therefore, the {\tt Matlab} function {\tt eig.m} does not yield reliable results  \cite{VanAssche2023,TNP2023}. Indeed, in most cases, the eigenvalues computed by {\tt eig.m}  are complex conjugate, while the   MOPs have only  real zeros.

To reduce the eigenvalue condition number of the Hessenberg matrix $H_n$, a new  diagonal balancing procedure has been recently introduced  in \cite{TNP2023}, whose   main idea consists of transforming $ H_n $ into a similar  matrix 
 \begin{equation}\label{eq:bal} \hat{H}_n :=S_n^{-1} H_n S_n 
\end{equation}
 having the same Hessenberg  structure,  with $S_n=\mbox{\tt diag}(s_1,\ldots,{s}_n)$ a diagonal matrix, such that
$\mbox{\tt triu}( \hat{H}_n,-1) $ is symmetric.
 % $ \hat{H}_n $ has the main tridiagonal part symmetric.
%\\
%The new balancing algorithm is implemented in the {\tt Matlab} function {\tt Dscale.m}, displayed in the Appendix, Table~\ref{tab:Alg1}.%the implementation of the scaling algorithm%  displayed ib the  can be 
\\
  After this similarity transformation, the  condition  of the eigenvalue problem for   $ \hat{H}_n$ is  drastically reduced with respect to that of  $H_n,$ as shown in Figure~\ref{figu:1}, where
the condition numbers of the eigenvalues of $ H_n$ and $ \hat{H}_n, $ with $ n=20$ and parameters listed in Table~\ref{tab:MOP} for all the considered MOPs,  are displayed.  
%%%%%%%%%%%%%%%%%%%%%%%%%%%%%%%%%%%%%%%%%%%%%%%%%%%%
\begin{table}
  \caption{MOPs parameters used for computing the eigenvalue condition number of  $ H_{20}$ and $ \hat{H}_{20}, $  shown in Figure~\ref{figu:1}.}
  \label{tab:MOP}
	 \centering
\begin{tabular}{|c|l|}%{|@{\hspace{.1cm}}l@{\hspace{.1cm}}|c|c|c|c|@{}c@{}|}
\hline
{\tt IC}  & $\qquad \quad\boldsymbol{\alpha}$\\ 
\hline
1 & $[-0.5,\;-0.2,\;0.4]^T$\\
\hline
2 & $[-0.5,\;0.5]^T$\\
\hline
3 & $[-0.5,\; 0.2,\;0.4]^T$\\
\hline
4 & $[0.2,\;0.5]^T$\\ %=.2;alpha2=.5
\hline
%5 &$[-2,\;-0.5,\;0.2]^T$ \\ %a=-2; alpha1=-.5;alpha2=.2;
%\hline
5 & $[0.5]^T$\\ 
\hline
6 &$[-0.5,\;0.5]^T$ \\ % alpha=-.5; nu=.5;
\hline
7 &$[0.5,\;-0.5]^T$ \\ % C=.5; vu=-.5;
\hline
8 &$[1,\;1,\;3,\;2]^T$ \\ % C=.5; vu=-.5;
\hline
9 &$[3,\;2.5,\;7.5]^T$ \\ % C=.5; vu=-.5;
\hline
\end{tabular}\end{table}
%%%%%%%%%%%%%%%%%%%%%%%%%%%%%%%%%%%%%%%%
% file Prova_MOP_poster1.m in VanAssche-old
\begin{figure}[h]
  \centering
\hspace*{-.8cm}\includegraphics[width=14.8cm,height=7cm]{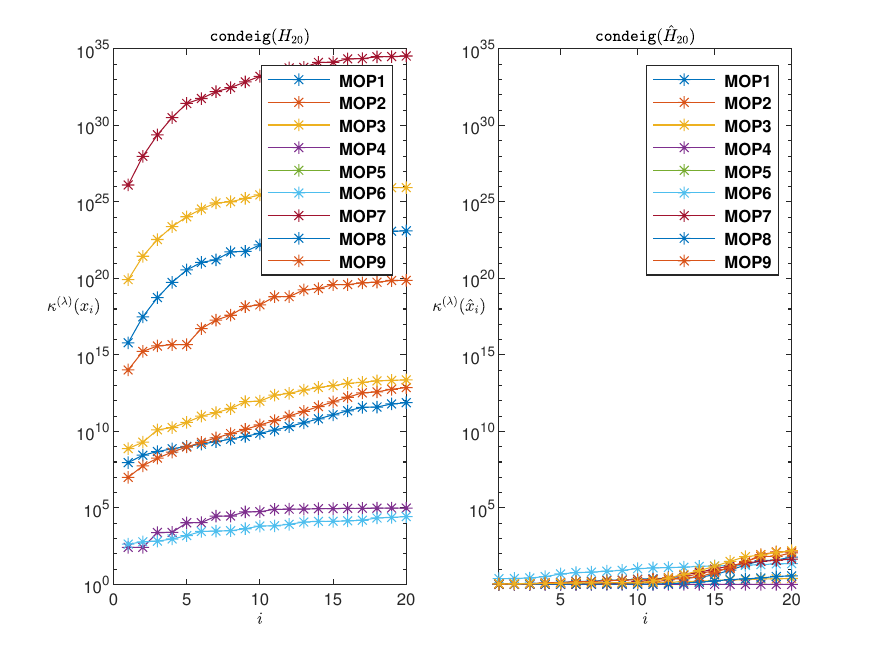}
  \caption{ Condition number of the eigenvalues of $H_{20}$ (left)  and   of the eigenvalues of $\hat{H}_{20}$ (right) for all the considered MOPs.}
 \label{figu:1}% \Description{Condition number of the eigenvalues of $H_{20}$ (left)  and   of the eigenvalues of $\hat{H}_{20}$ for all the considered MOPs.}
\end{figure}
Observe that the entries  of  the diagonal balancing matrix $ S_n=\mbox{\tt diag}(s_1,\ldots s_n)$ in (\ref{eq:bal}) , introduced in \cite{TNP2023}, grow as the factorial function, i.e., $ s_i \sim i! $  (see Table~\ref{tab:large1}, fourth column), 
for most MOPs. Hence, for $ i $ large enough, e.g.,  $ i > 170,$  $s_i> (2-2^{-52})2^{1023},$  the largest finite floating-point number in $\mbox{\rm IEEE}^{\mbox{\textregistered}}$
double precision. Therefore,  in such cases, the computation of $ s_i$ yields {\tt Inf} in {\tt Matlab}.
  
To overcome this drawback, we consider  here the same balancing technique applied  to the matrix $ H_n, $  but without explicitly computing the matrix $S_n. $ 
%Indeed, as shown in Section~\ref{subs:2}, 
Observe that we need to compute  only $  s_1 $ and $ s_ 2,$ i.e., the first two entries in the main diagonal  of $S_n,$   since they are involved in the  computation of  the weights of the simultaneous Gaussian quadrature rules (\ref{eq:wd}).
The {\tt Matlab} function {\tt  DscaleS2.m}, implementing the new balancing technique, is displayed in Table~\ref{tab:Alg1} in the Appendix. %, Section~\ref{subs:2}. 
%\end{remark}

\noindent
Since the considered MOPs are monic, i.e., $a_i=1, i=0,1,\ldots, n-1, $ then \cite{TNP2023} % the elements of the  vector $\boldsymbol{s}_n=[ s_1,\;s_2,\dots, s_n]^T $ are 
\begin{equation}\label{eq:ss}
s_i=\sqrt{c_1 c_2 \cdots c_i}, \quad i=1,\ldots, n.
\end{equation}
Furthermore,
denoted by
\begin{eqnarray*}
\hat{c}_i&=&\hat{h}_{i+1,i},\quad i=1,\ldots,n-1, \\
\hat{d}_i&=&\hat{h}_{i+1,i-1},\quad i=2,\ldots,n-1, 
\end{eqnarray*}
 defined $ s_0 \equiv 1,$ 
it turns out \cite{TNP2023},
\begin{eqnarray}
\hat{c}_i&=& c_i \frac{s_{i-1}}{s_{i}},\quad i=1,\ldots,n-1, \label{eq:hc0}\\
\hat{d}_i&=&d_i \frac{s_{i-2}}{s_{i}},\quad i=2,\ldots,n-1. \label{eq:hd}
\end{eqnarray}
%Moreover, %for large  $ i$,
For each MOP in Section~\ref{sect:MOP}, the asymptotic behavior  of $ c_i $ and $d_i, $  respectively denoted by  $c_{i \gg 0}$ and $d_{i \gg 0}, $ can be represented as the product of a constant times a power of $ i$, i.e.,
$$
 c_{i \gg 0} =\theta_c i^{\rho_c} \;\; \mbox{\rm and} \;\;  d_{i \gg 0} =\theta_d i^{\rho_d}, \quad 
\theta_c, \theta_d, \rho_c, \rho_d \in  \RR,
$$ 
except for $\mbox{\bf MOP}_3$, $\mbox{\bf MOP}_4$, %\ref{example5.1}, 
$\mbox{\bf MOP}_5$, where the even and odd cases need to be distinguished. Moreover, $c_{i \gg 0} $ and $ d_{i \gg 0}$ depend on $ \alpha_1$ and $ \alpha_2$ for $\mbox{\bf MOP}_3$ and $\mbox{\bf MOP}_4$, and on $ \beta$  for $\mbox{\bf MOP}_7$.
\\
The following Lemma holds.
\begin{lemma}\label{lemma:1}
Let  $c_{i}, $   $ d_{i } $, $\hat{c}_{i}, $   $\hat{d}_{i } $,  $i=0,1,\ldots, n-1,$ be, respectively,  the entries  of the  first and second subdiagonal of $ H_n $ and the entries  of the  first and second subdiagonal of $ \hat{H}_n= S_n^{-1} H_n S_n. $
Then
\begin{equation} \label{eq:hc}
\hat{c}_{i\gg 0}=\sqrt{c_i}, 
\qquad
%\end{equation}
%\begin{equation} \label{eq:hc}
\hat{d}_{i\gg 0}= \frac{d_i}{c_i}. 
\end{equation}
\end{lemma}
\begin{proof}
 By (\ref{eq:ss}), 
$$
%s_i= s_{i-1} \sqrt{\frac{c_i}{a_{i-1}}}= s_{i-1} \sqrt{\theta_c i^{\rho_c}}.
s_i= s_{i-1} \sqrt{c_i}= s_{i-1} \sqrt{\theta_c i^{\rho_c}}.
$$
Then
$$
s_i= \theta_c^{\frac{i}{2}} i!^{\frac{\rho_c}{2}}.
$$
Moreover, by \ref{eq:hc0},
%By \ref{A1.9} in Table~ 
$$
\hat{c}_{i \gg 0} = c_i \frac{s_{i-1}}{s_i} 
                  =  \theta_c i^{\rho_c} \frac{\theta_c^{\frac{i-1}{2}} (i-1)!^{\frac{\rho_c}{2}}}{\theta_c^{\frac{i}{2}} i!^{\frac{\rho_c}{2}}}
									= \sqrt{ \theta_c i^{\rho_c}}=\sqrt{c_{i\gg 0}},
$$
and,  by \ref{eq:hd},
$$
\hat{d}_{i \gg 0} = d_i \frac{s_{i-2}}{s_i} 
                  = \theta_d i^{\rho_d} \frac{\theta_c^{\frac{i-2}{2}} (i-2)!^{\frac{\rho_c}{2}}}{\theta_c^{\frac{i}{2}} i!^{\frac{\rho_c}{2}}}
									= \frac{ \theta_d}{\theta_c} i^{\rho_d-\rho_c}
									=\frac{d_{i\gg 0}}{c_{i\gg 0}}.
$$
\end{proof}
%By Lemma~\ref{lemma:1}, 
The asymptotic behavior of the coefficients $c_{i}, $   $ d_{i } $, $\hat{c}_{i}, $   $\hat{d}_{i } $,  $i=0,1,\ldots, n-1$, and $ s_i, $ $i=1,\ldots,n, $ are displayed in Table~\ref{tab:large1} for all the considered MOPs.
\begin{table*}
  \caption{Asymptotic behavior of the coefficients $c_{i}, $   $ d_{i } $, $\hat{c}_{i}, $   $\hat{d}_{i } $,  $i=0,1,\ldots, n-1$, and $ s_i, $ $i=1,\ldots,n. $}
  \label{tab:large1}
\begin{tabular}{|@{\hspace{.1cm}}c@{\hspace{.1cm}}|@{\hspace{.1cm}}c@{\hspace{.1cm}}|@{\hspace{.1cm}}c@{\hspace{.1cm}}|@{\hspace{.1cm}}c@{\hspace{.1cm}}|@{\hspace{.1cm}}c@{\hspace{.1cm}}|@{\hspace{.1cm}}c@{\hspace{.1cm}}|}%{|@{\hspace{.1cm}}l@{\hspace{.1cm}}|c|c|c|c|@{}c@{}|}
\hline
$\mbox{\bf MOP}$  & $ {  c_{i \gg 0}}$ &   ${  d_{i \gg 0}}$  &$ s_{i \gg 0}  $ & $ {   \hat{c}_{i \gg 0}}$  & $ {   \hat{d}_{i \gg 0}}$ \\ 
\hline
\hline
%\begin{tabular}{@{}c@{}}
$\mbox{\bf MOP}_1$
%Jacobi--Pi\~neiro 
%\\ $ \alpha_0=1,\; \alpha_1=.2,\; \alpha_2=.4$ 
%\end{tabular} 
& $ {   3\left(\frac{4}{27}\right)^2 }$ 
&$ {   \left(\frac{4}{27}\right)^3 }$  
& $ {   3^{i/2}\left(\frac{4}{27}\right)^i}$  
& $ {   \sqrt{ 3}\left(\frac{4}{27}\right) }$ 
&$ {  \frac{1}{3} \frac{4}{27}}$  \\
\hline
%\begin{tabular}{@{}l@{}}
 $\mbox{\bf MOP}_2$
%Laguerre \\ first kind 
 %\\ $ \alpha_0=-.5,\; \alpha_1=.5$ 
%\end{tabular} 
& $ {   3\left(\frac{i}{2}\right)^2 }$ & $ {  \left(\frac{i}{2}\right)^3 }$ & $ {  \left(\frac{\sqrt{3}}{2}\right)^i  i!}$  &$ {   \frac{\sqrt{3}}{2} i }$  &$ {  \frac{i}{6}} $  \\
\hline
%\begin{tabular}{@{}l @{}}
$\mbox{\bf MOP}_3$ 
%Laguerre \\ second kind
% \\ $ \alpha_0=1,\; \alpha_1=.1 \; \alpha_2=2;$ 
%\end{tabular} 
& $ {  \frac{i^2}{2}\frac{\alpha_1^2+\alpha_2^2}{\alpha_1^2\alpha_2^2} }$
 & $ { \frac{(-1)^i i^3 (\alpha_2-\alpha_1)}{2\left(\alpha_1^{2+(-1)^i}\alpha_2^{2-(-1)^i}\right)} }$ 
 & $ {  \left(\frac{\alpha_1^2+\alpha_2^2}{2\alpha_1^2\alpha_2^2}\right)^{\frac{i}{2}} i!}$
 &$ %{ }
\sqrt{\frac{\alpha_1^2+\alpha_2^2}{2(\alpha_1^2\alpha_2^2)}}i $  
 &$ {  \frac{ (-1)^ii(\alpha_2-\alpha_1)}{\alpha_1^{(-1)^i}\alpha_2^{(-1)^{i+1}}(\alpha_1^2+\alpha_2^2)} }$   \\
\hline
$\mbox{\bf MOP}_4$ 
%Hermite 
& $ {  \frac{i}{2}}$
& $   \frac{(-1)^i i(\alpha_1-\alpha_2)}{8}$
& $ {  \sqrt{ \frac{i!}{2^{i}}}}$
& $ {  \sqrt{\frac{i}{2}}}$
& $ {  \frac{(-1)^i (\alpha_1-\alpha_2)}{4}}$
\\
%\hline
%Jacobi--Laguerre  
%& $ {  \frac{i^2}{2}}$
%& $ {  \frac{1-(-1)^i}{2} \frac{ i^3}{2}}$
%& $ {  \sqrt{\frac{i!}{2^\frac{i}{2}}}}$
%& $ {  \frac{i}{\sqrt{2}}}$
%& $ {  \frac{1-(-1)^i}{2} i}$
 %\\
\hline
$\mbox{\bf MOP}_5$ 
% Laguerre--Hermite 
& $ {  \frac{i}{4}}$
& $ {  (-1)^{i+1} \frac{\sqrt{i^3}}{8}}$
& $ {  \sqrt{\frac{i!}{4^i}}}$
& $ {  \sqrt{\frac{i}{4}}}$
& $ {  (-1)^{i+1} \sqrt{\frac{i}{4}}}$
 \\
\hline
%\begin{tabular}{@{}l@{}}
$\mbox{\bf MOP}_6$
%Modified Bessel \\ function $K_{\nu}$
%\end{tabular} 
& $ {  3 i^4}$
& $ {  i^6}$
& $ {  3^{\frac{i}{2}} (i!)^2}$
& $ {  \sqrt{3}i^2}$
& $ {  \frac{1}{3}i^2}$
 \\
\hline
%\begin{tabular}{@{}l@{}}
 $\mbox{\bf MOP}_7$ 
%Modified Bessel \\ function $I_{\nu}$
% \end{tabular} 
& $ {  \frac{i^2}{\beta^2} }$
& $ {  \frac{i^2}{\beta^4} }$
& $ {  {\frac{i!}{\beta^{i}}} }$
& $ {  i}$
& $ {  \frac{i}{\beta^2}}$
\\
\hline
%\begin{tabular}{@{}l@{}}
$\mbox{\bf MOP}_8$
%Gauss' hypergeometric \\ function 
 %\end{tabular} 
& $ {  3 \left( \frac{4}{27}\right)^2 }$
& $ {  \left( \frac{4}{27}\right)^3 }$
& $ {  \left(\sqrt{3}\frac{4}{27}\right)^i }$
& $ {  \sqrt{3}\frac{4}{27}}$
& $ {  \frac{1}{3}\frac{4}{27}}$
\\
\hline
%\begin{tabular}{@{}l@{}}
$\mbox{\bf MOP}_9$
%confluent hypergeometric \\ function 
% \end{tabular} 
& $ {   \frac{280}{81} i^2}$
& $ {  3^{-3-3\frac{1+(-1)^i}{2}} 2^6 i^3}$
& $ {  \left(\frac{280}{81}\right)^\frac{i}{2} i! } $ % \left(\sqrt{3}\frac{4}{27}\right)^i }$
& $ {   \sqrt{\frac{280}{81}} i}$
& $ {  \frac{1}{35}3^{-3\frac{1+(-1)^i}{2}} 2^3 i}$
%& $ {  \frac{1}{35}\frac{3^{-3\frac{1+(-1)^i}{2}}}{ 2^3} i}$
 \\
\hline
\end{tabular}\end{table*}
We observe that, except for $\mbox{\bf MOP}_1$ and $\mbox{\bf MOP}_8$, whose coefficients
 $ {\displaystyle c_{i \gg 0}}, {\displaystyle d_{i \gg 0}},   {\displaystyle  \hat{c}_{i \gg 0}}$ and $ {\displaystyle  \hat{d}_{i \gg 0}}$  
 are independent of $ i, $ 
 % the multiple Jacobi--Pi\~{n}eiro polynomials and for the MOPs associated with the Gauss' hypergeometric  function, 
$$
\lim_{i \rightarrow \infty } \frac{  {\displaystyle  \hat{c}_{i \gg 0}}}{ {\displaystyle c_{i \gg 0}}} =0, \quad 
\lim_{i \rightarrow \infty } \frac{ {\displaystyle  \hat{d}_{i \gg 0}}}{{\displaystyle d_{i \gg 0}}} =0,
$$
and
$$
 {\displaystyle c_{i \gg 0}}  <  |{\displaystyle d_{i \gg 0}}|, \quad  {\displaystyle  \hat{c}_{i \gg 0}} > |{\displaystyle  \hat{d}_{i \gg 0}}|.
$$
Therefore, the new balancing technique has the effect of  reducing the size of the entries of the second lower subdiagonal of the Hessenberg matrix  $\hat{H}_n$. As a consequence, the condition numbers of the eigenvalues of $ \hat{H}_n $ are drastically reduced with respect to those of $ H_n$ (see Figure~\ref{figu:1}). 
%%%%%%%%%%%%%%%%%%%%%%%%%%%%%%%%%%%%%%%%%%%%%%%%%%%%%%%%%%%%%%%%%%%%%%%%%%%%%%%%%%%%
\section{Computation of the simultaneous Gaussian quadrature rule}\label{subs:2}%Computation of the eigenvalues and left and right eigenvectors of  $ \hat{H}_n $}
By %Lemma~\ref{lemma:1},
Theorem~\ref{th:quads},
  the eigenvalues of ${H}_n$ and the associated  left and right eigenvectors are needed for  computing the simultaneous Gaussian quadrature rule associated with MOPs.
Let $ \hat{x}_j, \;\hat{\boldsymbol{u}}^{(j)}, $  and $\hat{\boldsymbol{v}}^{(j)}, \; j=1,\ldots,n, $ be the eigenvalues and corresponding left and right eigenvectors of  $ \hat{H}_n, $ i.e.,
$$
 \hat{H}_n \hat{V}_n = \hat{V}_n \hat{X}_n, \quad  \hat{U}_n^T \hat{H}_n  = \hat{X}_n \hat{U}_n^T,
$$
with % \mbox{\rm with}
$$
  \hat{X}_n=\mbox{\rm diag}(\hat{x}_1,\ldots, \hat{x}_n),\; 
\hat{V}_n=[\hat{\boldsymbol{v}}^{(1)}, \ldots, \hat{\boldsymbol{v}}^{(n)}], \; \mbox{\rm and}\;
\hat{U}_n=[\hat{\boldsymbol{u}}^{(1)}, \ldots, \hat{\boldsymbol{u}}^{(n)}]. $$
 All the eigenvalues and corresponding left and right eigenvectors  of $ H_n$ can be computed applying the {\tt Matlab} function {\tt eig.m} to the better conditioned matrix $ \hat{H}_n $,  with $\mathcal{O}(n^3)$
computational complexity and $\mathcal{O}(n^2)$ memory.
Since 
$ \hat{H}_n=S_n^{-1} H_n S_n, $ then $  {U}_n= S_n^{-1} \hat{U}_n, $ and  $  {V}_n= S_n \hat{V}_n. $
\\
 By (\ref{eq:weights}), for $j=1,\ldots,n,$
%$$
%\begin{array}{@{}ll}
%\begin{array}{l}
%{\displaystyle \omega_j^{(1)}}=\frac{\displaystyle  v_1^{(j)} f_{1,1}u_1^{(j)}}{\displaystyle \boldsymbol{u}^{(j)^T} \boldsymbol{v}^{(j)}}=
%%\frac{\displaystyle  s_1\hat{v}_1^{(j)} f_{1,1}s_1^{-1}\hat{u}_1^{(j)}}{\displaystyle \hat{\boldsymbol{u}}^{(j)^T} S_n S_n^{-1}\hat{\boldsymbol{v}}^{(j)}}
%\frac{\displaystyle  \hat{v}_1^{(j)} f_{1,1}\hat{u}_1^{(j)}}{\displaystyle \hat{\boldsymbol{u}}^{(j)^T} \hat{\boldsymbol{v}}^{(j)}},\\
%{\displaystyle \omega_j^{(2)}}=\frac{\displaystyle v_1^{(j)}\left(  f_{2,1}u_1^{(j)}+f_{2,2}u_2^{(j)}\right)}{\displaystyle  \boldsymbol{u}^{(j)^T} \boldsymbol{v}^{(j)}}=
%\frac{\displaystyle s_1v_1^{(j)}\left(  f_{2,1}s_1^{-1}\hat{u}_1^{(j)}+f_{2,2}s_2^{-1}\hat{u}_2^{(j)}\right)}{\displaystyle \hat{\boldsymbol{u}}^{(j)^T} \hat{\boldsymbol{v}}^{(j)}},
%\end{array}&
%\quad j=1,\ldots,n.
%\end{array}
%$$
\begin{equation}\label{eq:wd}
%\begin{array}{l}
\begin{array}{@{}l@{}}
{\displaystyle \omega_j^{(1)}}=\frac{\displaystyle  v_1^{(j)} f_{1,1}u_1^{(j)}}{\displaystyle \boldsymbol{u}^{(j)^T} \boldsymbol{v}^{(j)}} =
\frac{\displaystyle  \hat{v}_1^{(j)} f_{1,1}\hat{u}_1^{(j)}}{\displaystyle \hat{\boldsymbol{u}}^{(j)^T} \hat{\boldsymbol{v}}^{(j)}},\\
{\displaystyle \omega_j^{(2)}}=\frac{\displaystyle v_1^{(j)}\left(  f_{2,1}u_1^{(j)}+f_{2,2}u_2^{(j)}\right)}{\displaystyle  \boldsymbol{u}^{(j)^T} \boldsymbol{v}^{(j)}} =
\frac{\displaystyle s_1\hat{v}_1^{(j)} \left(  f_{2,1}s_1^{-1}\hat{u}_1^{(j)}+f_{2,2}s_2^{-1}\hat{u}_2^{(j)}\right)}{\displaystyle \hat{\boldsymbol{u}}^{(j)^T} \hat{\boldsymbol{v}}^{(j)}}.
\end{array}
%\qquad
%\end{array}
\end{equation}

Hence, the computation of the nodes and the set of weights of the simultaneous Gaussian quadrature rule relies only on the eigenvalues of $ \hat{H}_n $ and corresponding left and right  normalized eigenvectors,  on  $ f_{1,1},  f_{2,1} $ and $ f_{2,2} $  (\ref{eq:D}),
%$$\left[\begin{array}{cc} \end{array} \right],$$
 and on the first two diagonal entries of $ S_n. $ 
%\quad 
%{\displaystyle \omega_j^{(2)}}=\frac{\displaystyle v_1^{(j)}\left(  f_{2,1}u_1^{(j)}+f_{2,2}u_2^{(j)}\right)}{\displaystyle  \boldsymbol{u}^{(j)^T} \boldsymbol{v}^{(j)}},
%%\end{array}
%\quad j=1,\ldots,n,

Here, we propose a  more efficient approach to compute the eigenvalues and corresponding left and right eigenvectors of  $ \hat{H}_n, $   with $\mathcal{O}(n^2)$
computational complexity and $\mathcal{O}(n)$ memory, based on the Ehrlich--Aberth method \cite{Aberth,Bini,Ehrlich1}.

Given an initial guess %  approximation 
$\hat{\boldsymbol{x}}^{(0)^T}=\left[
\begin{array}{cccc} 
\hat{x}_1^{(0)}, & \hat{x}_2^{(0)}, & \cdots, \hat{x}_{n-1}^{(0)}, & \hat{x}_n^{(0)} 
\end{array}\right]^T$
of all the eigenvalues of $\hat H_n$, 
they can be simultaneously computed by the Ehrlich--Aberth iteration \cite{Aberth,Bini,Ehrlich1},
$$
\hat{x}_j^{(\ell+1)}=\hat{x}_j^{(\ell)}-\frac{\frac{\displaystyle \hat{p}_n(\hat{x}_j^{(\ell)})}{\displaystyle \hat{p}_n'(\hat{x}_j^{(\ell)})}}{1-\frac{\displaystyle \hat{p}_n(\hat{x}_j^{(\ell)})}{\displaystyle \hat{p}_n'(\hat{x}_j^{(\ell)})} {\displaystyle \sum_{^{k=1}_{k\ne j}}^{n}{\displaystyle \frac{1}{\hat{x}_j^{(\ell)}-\hat{x}_k^{(\ell)}}}}},\qquad j=1,\ldots,n,
$$
where  $ \hat{p}_i(x), \; i=0,1\ldots,n, $ are non--monic MOPs satisfying the recurrence relation whose coefficients are those of the balanced  matrix $\hat{H}_{n,n+1}, $
\begin{equation} \label{eq:MOPT} 
\begin{array}{@{}l@{}l@{}}
\left\{ \begin{array}{@{}l@{}}
\hat{p}_{-2}(x)=0,  \\
\hat{p}_{-1}(x)=0,  \\
%x\hat{p}_{i}(x) = \hat{c}_{i-1} \hat{p}_{i+1}(x)  + \hat{b}_i \hat{p}_{i}(x)  + \hat{c}_i \hat{p}_{i-1}(x)  + \hat{d}_i \hat{p}_{i-2}(x), \\
\hat{c}_{i-1} \hat{p}_{i+1}(x)= (x-  \hat{b}_i) \hat{p}_{i}(x)  -\hat{c}_i \hat{p}_{i-1}(x)  - \hat{d}_i \hat{p}_{i-2}(x),
\end{array}\right. & i=0, \ldots, n-1,
\end{array}
\end{equation}
with $\hat{c}_i=\hat{a}_{i-1}, \; i=1,\dots,n-1. $
Hence,  $\hat{p}_n(\hat{x}_j^{(i)}) $ and $ \hat{p}'_n(\hat{x}_j^{(i)})$ need to be computed at  iteration $ \ell+1$ of the Ehrlich--Aberth method.

To this end,
(\ref{eq:MOPT})  is written in matrix form,  obtaining
%$$ 
%\begin{array}{@{}l}
%\hat{F}_n(x)  \hat{\boldsymbol{p}}_n(x):= 
%\left(\hat{H}_{n,n+1}-x[I_n|O_{1,n}]\right)\hat{\boldsymbol{p}}_n(x) \label{eq:null1} \\
%= \left[\begin{array} 
%{@{\hspace{1.5mm}}c@{\hspace{1.5mm}}c@{\hspace{1.5mm}}c@{\hspace{1.5mm}}c@{\hspace{1.5mm}}c@{\hspace{1.5mm}}c@{\hspace{1.5mm}}c@{\hspace{1.5mm}}c@{\hspace{1.5mm}}c@{\hspace{1.5mm}}c@{\hspace{1.5mm}}}
 %\bt_0-x& \at_0 & 0 & 0 & 0 & 0 & \ldots & 0 & 0 \\
%\ct_1 & \bt_1-x & \at_1 & 0 & 0 & 0 & \ldots & 0 & 0 \\
%\dt_2 & \ct_2 & \bt_2-x & \at_2 & 0 & 0 & \ldots & 0 & 0 \\
%0 & \dt_3 & \ct_3 & \bt_3-x & \at_3 & 0 & \ldots & 0 & 0 \\
%0 & 0 & \dt_4 & \ct_4 & \bt_4-x & \at_4 & \ldots & 0 & 0 \\
%\vdots &  &  & \ddots & \ddots & \ddots & \ddots & \vdots & \vdots \\
%0 & \ldots & 0 & 0 & \dt_{n-2} & \ct_{n-2} & \bt_{n-2}-x & \at_{n-2}& 0 \\
%0 & \ldots & 0 & 0 & 0 & \dt_{n-1} & \ct_{n-1} & \bt_{n-1}-x &   \at_{n-1}\\
%\end{array} \right]%\in  \RR^{n \times (n+1)}.
%\left[\begin{array}{@{}c@{}} 
%\Pt_0(x) \\ \Pt_1(x)\\  \Pt_2(x)\\  \Pt_3(x)\\  \Pt_4(x)\\ \vdots \\  \Pt_{n-1}(x)\\  \Pt_{n}(x) 
%\end{array}\right]= {\bf 0}_n. 
%\nonumber 
%%&:= & \hat{F}(x)  \hat{\boldsymbol{p}}_n(x)=0, 
%\end{array}
%$$
%%%%%%%%%%%%%%%%%%%%%%%%%%
$$ 
\begin{array}{@{}l}
\hat{F}_n(x)  \hat{\boldsymbol{p}}_n(x):= 
\left(\hat{H}_{n,n+1}-x[I_n|\boldsymbol{o}_{n}]\right)\hat{\boldsymbol{p}}_n(x) \label{eq:null1} \\
= \left[\begin{array} 
{@{\hspace{1.5mm}}c@{\hspace{1.5mm}}c@{\hspace{1.5mm}}c@{\hspace{1.5mm}}c@{\hspace{1.5mm}}c@{\hspace{1.5mm}}c@{\hspace{1.5mm}}c@{\hspace{1.5mm}}c@{\hspace{1.5mm}}c@{\hspace{1.5mm}}}
 \bt_0-x& \at_0 & 0 & 0 & 0  & \ldots & 0 & 0 \\
\ct_1 & \bt_1-x & \at_1 & 0  & 0 & \ldots & 0 & 0 \\
\dt_2 & \ct_2 & \bt_2-x & \at_2  & 0 & \ldots & 0 & 0 \\
0 & \dt_3 & \ct_3 & \bt_3-x & \at_3  & \ldots & 0 & 0 \\
%0 & 0 & \dt_4 & \ct_4 & \bt_4-x & \at_4 & \ldots & 0 & 0 \\
\vdots &    & \ddots & \ddots & \ddots & \ddots & \vdots & \vdots \\
0 & \ldots & 0  & \dt_{n-2} & \ct_{n-2} & \bt_{n-2}-x & \at_{n-2}& 0 \\
0 & \ldots & 0  & 0 & \dt_{n-1} & \ct_{n-1} & \bt_{n-1}-x &   \at_{n-1}\\
\end{array} \right]%\in  \RR^{n \times (n+1)}.
\left[\begin{array}{@{}c@{}} 
\Pt_0(x) \\ \Pt_1(x)\\  \Pt_2(x)\\  \Pt_3(x)\\ 
%\Pt_4(x)\\
\vdots \\  \Pt_{n-1}(x)\\  \Pt_{n}(x) 
\end{array}\right]= \boldsymbol{o}_{n}. 
\nonumber 
%&:= & \hat{F}(x)  \hat{\boldsymbol{p}}_n(x)=0, 
\end{array}
$$
%%%%%%%%%%%%%%%%%%%%%%%%%%%%%
Therefore,  $\hat{\boldsymbol{p}}_n(x)$ belongs to the null--space of $ \hat{F}_n(x), $  a full row rank matrix.
%Given $ {x} \in \RR, $ $ \Pt_{n}({x}) $ can be evaluated by computing  the vector belonging to the null--space of $ \hat{F}({x}). $ 
Such a vector  is computed by applying a sequence of $n $ Givens rotations, $G_i \in  \RR^{(n+1)\times (n+1)},\; i=1,\ldots,n, $ to the right of $\hat{F}_n(x) $ %obtaining a lower triangular matrix 
\cite{TNP2023},
 $$
\hat{F}_n(x)Q_n = [ L_n,\;\boldsymbol{o}_{n}],
$$
with $ L_n \in  \RR^{n\times n}$ lower triangular, % $ O_n={\tt zeros}(n,1),$ 
and $Q_n = G_1^T G_2^T \cdots G_n^T. $ 
Then, the last column of $ Q_n $ spans the null--space of  $\hat{F}_n(x). $

%Moreover, 
In order to compute $ \Pt'_{n}(\tilde{x}) $,  we differentiate (\ref{eq:null1}), obtaining
\begin{equation}\label{eq:deriv}
\hat{F}_n(x)  \hat{\boldsymbol{p}}'_n(x)+\hat{F}'(x)  \hat{\boldsymbol{p}}_n(x)=\boldsymbol{o}_{n}.
\end{equation}
Since 
$
\hat{F}'(x) =\left[\begin{array}{@{}c|c@{}}-I_n & \boldsymbol{o}_{n} \end{array} \right]
$
and $ \hat{p}'_0(x)=0,$  %then (\ref{eq:deriv}) becomes
then $ [\hat{p}'_1(x), \hat{p}'_2(x),\ldots,\hat{p}'_{n-1}(x), \hat{p}'_{n}(x)]^T $ is the solution of the lower triangular linear system
$$
\left[\begin{array} 
{@{\hspace{1.5mm}}c@{\hspace{1.5mm}}c@{\hspace{1.5mm}}c@{\hspace{1.5mm}}c@{\hspace{1.5mm}}c@{\hspace{1.5mm}}c@{\hspace{1.5mm}}c@{\hspace{1.5mm}}c@{\hspace{1.5mm}}c@{\hspace{1.5mm}}}
 \at_0 & 0 & 0 & 0 & 0 & \ldots & 0  \\
 \bt_1-x & \at_1 & 0 & 0 & 0 & \ldots & 0  \\
 \ct_2 & \bt_2-x & \at_2 & 0 & 0 & \ldots & 0 \\
 \dt_3 & \ct_3 & \bt_3-x & \at_3 & 0 & \ldots & 0  \\
% 0 & \dt_4 & \ct_4 & \bt_4-x & \at_4 & \ldots & 0 & 0 \\
\vdots &     \ddots & \ddots & \ddots & \ddots & \vdots & \vdots\\
 0 &\ldots &   \dt_{n-2} & \ct_{n-2} & \bt_{n-2}-x & \at_{n-2}& 0 \\
 0 &\ldots &   0 & \dt_{n-1} & \ct_{n-1} & \bt_{n-1}-x &   \at_{n-1}\\
\end{array} \right]%\in  \RR^{n \times (n+1)}.
\left[\begin{array}{@{}c@{}} 
 \Pt'_1(x)\\  \Pt'_2(x)\\  \Pt'_3(x)\\  \Pt'_4(x)\\  \vdots \\  \Pt'_{n-1}(x)\\  \Pt'_{n}(x) 
\end{array}\right]=
\left[\begin{array}{@{}c@{}} 
\Pt_0(x) \\ \Pt_1(x)\\  \Pt_2(x)\\  \Pt_3(x)\\  \vdots \\  \Pt_{n-2}(x)\\  \Pt_{n-1}(x) 
\end{array}\right].
%\left[\begin{array}{@{}c@{}} 
 %\Pt_1(x)\\  \Pt_2(x)\\  \Pt_3(x)\\  \Pt_4(x)\\  \Pt_5(x)\\ \vdots \\  \Pt_{n-1}(x) 
%\end{array}\right].
$$
The {\tt Matlab} function  {\tt one\_step\_Newton.m} computing $\Pt_{n}(x)$ and $ \Pt'_{n}(x)$ with %${O} (n) $ 
$ 35 n$ flops,  making use of  the {\tt Matlab} function {\tt givens.m}, % floating point operation
 is displayed in Table~\ref{tab:AlgG} of the Appendix. % The  is called in this function.

The  sequence of approximations generated by the Ehrlich--Aberth method converges cubically to the eigenvalues of $\hat{H}_n$, or even faster if  the implementation is in the
Gauss--Seidel style, since the eigenvalues  are simple. In practice, as noticed in \cite{Bini}, the
Ehrlich--Aberth iteration exhibits good global convergence properties, though no theoretical
results seem to be known about global convergence.

The main requirements for the success of the Ehrlich--Aberth method are
 a fast, robust, and stable computation of the Newton correction $ \Pt_{n}(x)/\Pt'_{n}(x),$ and 
 a good set of initial approximations for the zeros, $\hat{\boldsymbol{x}}^{(0)},$ so that
the number of iterations needed for convergence is not too large.
\\
Since the eigenvalues of $ \hat{H}_n$ are real,  different approaches can be taken into account to compute the initial guess of the vector
$  \hat{\boldsymbol{x}}^{(0)}.$
Here, we consider two algorithms requiring  $\mathcal{O}(n^2) $ floating point operations and   $\mathcal{O}(n)$ memory:\footnote{We have written the {\tt Matlab} function {\tt gausq2.m},  a modified version of the {\tt fortran} routine {\tt  imtql2.f } from  {\tt eispack} \cite{TQL1,eispack}, to compute the eigenvalues of a symmetric tridiagonal matrix $ T_n$ , given only its main diagonal $\mbox{\tt diag} (T_n)$ and the subdiagonal $ \mbox{\tt diag} (T_n,-1)$: {\tt [$x$]=gausq2($\mbox{\tt diag} (T_n),\; \mbox{\tt diag} (T_n,-1)$).}}
\begin{enumerate}
\item
 {\tt $\hat{\boldsymbol{x}}^{(0)} $=gausq2($\mbox{\tt diag}(\hat{H}_n),\;\mbox{\tt diag} (\hat{H}_n,-1)$)},
 %$ \hat{\boldsymbol{x}}^{(0)} =\mbox{\tt eig}\left(\mbox{\tt triu}(  \hat{H}_n,-1)\right), $ 
i.e., $\hat{\boldsymbol{x}}^{(0)} $ is the vector of the eigenvalues of the symmetric tridiagonal matrix obtained setting $ \hat{d}_i=0, \; i=2,\ldots, n-1,$ in $\hat{H}_n.  $ Computational complexity: $\mathcal{O}(n^2)$, memory: $\mathcal{O}(n) $;
\item
\begin{enumerate}[label={ \alph*}.]
\item
reduction of  $ \hat{H}_n$ to  a similar  nonsymmetric tridiagonal matrix  $ \hat{T}_n$, by using elementary transformations.
Computational complexity:  $\frac{7}{2} n^2 $ flops;
\item
reduction of  $ \hat{T}_n$ to  a similar  symmetric tridiagonal matrix  $ T_n= \hat{D}_n^{-1} \hat{T}_n \hat{D}_n $,  with $\hat{D}_n$ a diagonal matrix.   Computational complexity: $5 n $ flops;
\item 
%computation of the 
{\tt $\hat{\boldsymbol{x}}^{(0)} $=gausq2($\mbox{\tt diag}({T}_n),\; \mbox{\tt diag}({T}_n,-1)$)}.
%$ \hat{\boldsymbol{x}}^{(0)} =\mbox{\tt eig}\left({T}_n\right).  $ 
Computational complexity: $\mathcal{O}(n^2) $, memory: $\mathcal{O}(n). $
\end{enumerate}
\end{enumerate}
While the first approach yields an approximation of the eigenvalues of $ \hat{H}_n, $  and then of $ H_n, $  the second one
provides  the eigenvalues of  $ \hat{H}_n $ as  $  \hat{\boldsymbol{x}}^{(0)}$, if computed in exact arithmetic.

The analysis  of the implementation of the second approach in floating point arithmetic has been described in \cite{TNP2023}. 
In practice, this approach works in a stable way for all the  MOPs listed in  Section~\ref{sect:MOP} and the 
Ehrlich--Aberth method converges in one iteration.
Therefore, approach  (2) is adopted in order to compute the initial vector $   \hat{\boldsymbol{x}}^{(0)}. $
  We observed that,  even though  $ \mbox{\tt diag}( \hat{H}_n,-2) $ has negative entries, %  in the  subdiagonal %$\hat{\boldsymbol{d}},$
 the entries of the subdiagonal and superdiagonal of $ \hat{T}_n $ are always positive.  Therefore, the similar symmetric matrix $ T_n $ can be computed without requiring complex arithmetic.

The {\tt Matlab} function {\tt tridEHbackwardV.m},  implementing the reduction of  $ \hat{H}_n$ into  a similar  nonsymmetric tridiagonal matrix  $ \hat{T}_n$ (step (2) (a)), is displayed in Table~\ref{tab:trid} of the Appendix.
\\
The {\tt Matlab} function {\tt DScaleSV2.m},  implementing the reduction of  $ \hat{T}_n$ into  a similar  symmetric tridiagonal matrix  $ {T}_n$ (step (2) (a)), is a simplification of the function {\tt DScaleSV2.m} and it is displayed in Table~\ref{tab:Alg1.0} of the Appendix.
\\          %%%%%%%%%%%%%%%%%%%%%%%%%
At iteration $ \ell+1, $ we consider  
$$\hat{x}_j^{(\ell+1)} \quad \mbox{\rm and}\quad 
\hat{\boldsymbol{v}}^{(j)}= \left[\begin{array}{@{\hspace{.4mm}}c@{\hspace{.4mm}}c@{\hspace{.4mm}}c@{\hspace{.4mm}}c@{\hspace{.4mm}}c@{\hspace{.4mm}}c@{\hspace{.4mm}}} 
\Pt_0(x), & \Pt_1(x),& \Pt_2(x),&   \cdots, &  \Pt_{n-2}(x),&  \Pt_{n-1}(x) 
\end{array}\right]^T,
%\hat{\boldsymbol{v}}^{(j)}=\frac{\displaystyle \left[\begin{array}{@{\hspace{.4mm}}c@{\hspace{.4mm}}c@{\hspace{.4mm}}c@{\hspace{.4mm}}c@{\hspace{.4mm}}c@{\hspace{.4mm}}c@{\hspace{.4mm}}} 
%\Pt_0(x), & \Pt_1(x),& \Pt_2(x),&   \cdots, &  \Pt_{n-2}(x),&  \Pt_{n-1}(x) 
%\end{array}\right]^T}
%{\left\|\displaystyle \left[\begin{array}{@{\hspace{.4mm}}c@{\hspace{.4mm}}c@{\hspace{.4mm}}c@{\hspace{.4mm}}c@{\hspace{.4mm}}c@{\hspace{.4mm}}c@{\hspace{.4mm}}} 
%\Pt_0(x), & \Pt_1(x),& \Pt_2(x),&   \cdots, &  \Pt_{n-2}(x),&  \Pt_{n-1}(x) 
%\end{array}\right]^T\right\|_2},
$$ 
respectively,  as an eigenvalue and corresponding %normalized 
right eigenvector  of $ \hat{H}_n$ if 
$$ | \Pt_{n}(\hat{x}_j^{(\ell+1)})| \le tol_1. $$ 
The corresponding  left eigenvector is computed applying a sequence of $ n-1$ Givens rotations 
$$G_i  =
\left[\begin{array}{cccc}
I_{i-1} &      &     &     \\
        & \gamma_i  & \sigma_i &     \\
				& -\sigma_i & \gamma_i &     \\
				&      &     & I_{n-i-1} 
				\end{array}
				\right]
$$ to the left of  $ \hat{H}_n -\hat{x}_j^{(\ell+1)}I_n$, a numerically singular matrix, such that 
$$
G_1^T G_2^T \cdots G_{n-1}^T \left(\tilde{H}_n -\hat{x}_j^{(\ell+1)}I_n\right) 
$$
has the upper diagonal and the $(1,1) $ entry annihilated. Therefore,
$$
\left[\begin{array}{@{}c@{}}
\gamma_1 \\
-\sigma_1 \gamma_2 \\
\sigma_1 \sigma_2 \gamma_3 \\
\vdots \\
(-1)^{n-2}  \displaystyle \prod_{k=1}^{n-2}\sigma_k \gamma_{n-1}\\
(-1)^{n-1} \displaystyle \prod_{k=1}^{n-1}\sigma_{k}
\end{array}
\right],
$$
 i.e., the first column of    $G_{n-1} G_{n-2} \cdots G_2 G_1, $ is the normalized left eigenvector $\hat{\boldsymbol{v}}^{(j)}.$
The {\tt Matlab} function {\tt left\_eigvect.m},  computing the normalized left eigenvector of $\hat{H}_n, $  is displayed in Table~\ref{tab:eigvect}.

%\\
%The {\tt Matlab} implementation reducing   $ \hat{T}_n$ to  a similar  symmetric tridiagonal matrix  $ T_n= \hat{D}_n^{-1} \hat{T}_n \hat{D}_n $ (step (2) (b)), is obtained from   the {\tt Matlab} function {\tt Dscale.m} described in Table~\ref{tab:Alg1.0}.
%% removing line \ref{A1.10} since $ \hat{T}_n$ is a nonsymetric tridiagonal matrix.

The {\tt Matlab}  function {\tt EA\_method.m}, implementing the  Ehrlich--Aberth method, is described in Table~\ref{tab:AlgEA}.

The {\tt Matlab} implementation of the simultaneous Gaussian quadrature rules, called  {\tt GaussMOP.m}, is described in Table~\ref{tab:Alg0}.
%%%%%%%%%%%%%%%%%%%%%%%%%%%%%%%%%%%%%%%%%%%%%%%%%%%%%%%%%%%%%%%%%%%%%%%%%%%%%%%%%%%%%%%%%%%%%%
\begin{remark}
In case  $ r\ge 2, $  the system of monic MOPs satisfies a  $ (r+2)$--term recurrence relation
\begin{equation}\label{eq:MOPr}
x p_n(x)= p_{n+1}(x)+\sum_{j=0}^{r} a_{n,j} p_{n-j}(x), \quad n\ge 0. 
\end{equation}
Writing (\ref{eq:MOPr}) in matrix form, we obtain 
$$
H_n 
\left[\begin{array}{c} 
p_0(x) \\ p_1 (x) \\ \vdots \\  p_{n-1}(x)
\end{array}\right] +
\left[\begin{array}{c} 
0 \\ \vdots \\ 0 \\p_n (x) 
\end{array}\right] =
x
\left[\begin{array}{c} 
p_0(x) \\ p_1 (x) \\ \vdots \\  p_{n-1}(x)
\end{array}\right],
$$
with
$$
H_n=
\left[
\begin{array}{ccccccccc}
a_{0,0} & 1  &    &        &  \\ 
a_{1,1} & a_{1,0} & 1  &     &        &  \\ 
a_{2,2} & a_{2,1} & a_{2,0} & 1  &     &        &  \\ 
\vdots  &\ddots  & \ddots &\ddots &\ddots\\
a_{r,r} & a_{r,r-1} & \cdots  &  a_{r,1} &  a_{r,0} & 1  &     &         \\
        & a_{r+1,r} & a_{r+1,r-1} & \cdots  &  a_{r+1,1} &  a_{r+1,0} & \ddots  &               \\
				&           & \ddots      & \ddots  & \ddots      & \ddots & \ddots      &  1 \ \\
				&           &             & a_{n-1,r} & a_{n-1,r-1} & \cdots  &  a_{n-1,1} &  a_{n-1,0} 
\end{array}
\right].
$$
Hence, similarly  to the case $ r=2, $  the simultaneous Gaussian quadrature rule for $ r> 2 $  can be retrieved from the eigenvalues and corresponding left and right eigenvectors of $ H_n$  \cite{VanAssche2005}.
\\
A sketch of an algorithm for computing   the simultaneous Gaussian quadrature rule for $ r> 2 $ can be summarized in the following steps:
\begin{enumerate}
\item construct the Hessenberg matrix $ H_n$;
\item compute $ \hat{H}_n = S_n^{-1} H_n S_n,$ with $ S_n $ a diagonal matrix such that
$\mbox{\tt triu}( \hat{H}_n,-1) $ is  symmetric;
\item compute the left and right eigendecompsition of $\hat{H}_n:$ $  \hat{U}_n^T \hat{H}_n= \hat{\Lambda}_n\hat{U}_n^T $ and $   \hat{H}_n\hat{V}_n=\hat{V}_n \hat{\Lambda}_n;$
\item retrieve the nodes and the weights of the simultaneous Gaussian quadrature rule from  $ \hat{\Lambda}_n, \hat{U}_n, \hat{V}_n $ and $ S_n. $
\end{enumerate}
\end{remark}

\section{Numerical Tests}
\label{sect:NE}
In this section we report some numerical tests %a numerical example 
performed in   {\tt Matlab R2022a}, with machine precision  $  \varepsilon \approx 2.22\times 10^{-16}.$ For all considered MOPs, the results of
{\tt GaussMOP.m}
 %the  proposed simultaneous Gaussian quadrature rules
 with $ n $ nodes, for $ n=10,20,\ldots, 90, 100, $
are compared to those obtained by using   the {\tt Matlab} function {\tt integral.m} % and the  {\tt Matlab}  variable precision  function {\tt vpaintegral.m}.
%The former is applied  in two ways: 
 %without setting the absolute error,  denoted by $ I_1, $ and   imposing the   absolute error equal to $ 10^{-13} $ , denoted by  $I_2.$  
%%The latter is used setting the precision to   $ 10^{-20}$ and converted to double precision by means of the {\tt Matlab} function {\tt double.m}.
%The function {\tt integral.m} is
 applied  in two ways: 
 without setting the absolute error,  denoted by $ I_1, $ and   imposing the   absolute error equal to $ 10^{-13} $ , denoted by  $I_2.$ 
%The function {\tt vpaintegral.m} is used setting the precision to   $ 10^{-20}$ and converted to double precision by means of the {\tt Matlab} function {\tt double.m}.
The same integrals were also computed by   the {\tt Mathematica 13.0} function {\tt Integrate}, setting the precision to $ 80 $ digits. These values, rounded to floating point numbers by means of the {\tt matlab} function {\tt double.m}, % and but  we observed that some values were different from those obtained by  {\tt vpaintegral.m}. Here, the results obtained  by {\tt vpaintegral.m}
 are assumed  to be the exact ones and compared with the results of  {\tt GaussMOP.m}, $ I_1, $ and $ I_2.$%  since the sequence of the numerical approximations of the integrals yielded by {\tt GaussMOP.m} converges to those ones, as $ n $ increases. 
\begin{example}\label{ex:Num1} In this example, {\tt GaussMOP.m}
 is used to simultaneously compute the integrals
$$\int_{\Delta^{(i)}}f(x) w^{(1)}(x)dx,\quad \int_{\Delta^{(i)}}f(x) w^{(2)}(x)dx,  $$
where  the integrand function is  $ f(x)= x e^{-x},  $ the same function considered in \cite{VanAssche2023}, and the input parameters for all MOPs are those displayed in Table~\ref{tab:MOP}.

The absolute errors of the integrals   computed by {\tt GaussMOP.m} are displayed  in Table~\ref{tab:ExNum1.0} and Table~\ref{tab:ExNum1.1} for  $ w^{(1)}$ and  $w^{(2)}$, respectively. Moreover,  in the last two rows of these tables, the results obtained applying the {\tt Matlab} function {\tt integral.m} %without setting the absolute error (last but one row) and with  $ 10^{-13} $  absolute error, denoted respectively by 
($I_1$ and $I_2 $),  are reported.

Observe that, in some  cases, the functions $I_1$ and $I_2 $ %{\tt integral}
 yield  {\tt NaN} as results, displaying the message ``{\tt Warning: Inf or NaN value encountered}''. 
\begin{table*}
  \caption{Absolute errors of the integrals, with weight $ w^{(1)}$,  computed by {\tt GaussMOP.m},  for $n=10, 20,\ldots, 100, $ and by $ I_1$ and $I_2$.}
		\centering
  \label{tab:ExNum1.0}
		\footnotesize
\begin{tabular}{@{}|@{\hspace{.4mm}}r@{\hspace{.4mm}}|@{\hspace{.4mm}}c@{\hspace{.4mm}}|@{\hspace{.4mm}}c@{\hspace{.4mm}}|@{\hspace{.4mm}}c@{\hspace{.4mm}}|@{\hspace{.4mm}}c@{\hspace{.4mm}}|@{\hspace{.4mm}}c@{\hspace{.4mm}}|@{\hspace{.4mm}}c@{\hspace{.4mm}}|@{\hspace{.4mm}}c@{\hspace{.4mm}}|@{\hspace{.4mm}}c@{\hspace{.4mm}}|@{\hspace{.4mm}}c@{\hspace{.4mm}}|@{}}%{|@{\hspace{.1cm}}l@{\hspace{.1cm}}|c|c|c|c|@{}c@{}|}
\hline
\multicolumn{10}{|c|}{$\int_{\Delta^{(i)}}f(x) w^{(1)}(x)dx, \quad f(x)= xe^{-x}$} \\
\hline
$n$ &$\mbox{\bf MOP}_1$ &$\mbox{\bf MOP}_2$&$\mbox{\bf MOP}_3$ &$\mbox{\bf MOP}_4$ &$\mbox{\bf MOP}_5$ &$\mbox{\bf MOP}_6$ &$\mbox{\bf MOP}_7$ &$\mbox{\bf MOP}_8$&$\mbox{\bf MOP}_9$ \\ 
\hline
 $10$&            $ 0        $&$     3.23 (-9)$&$    7.17 (-4)$&$    5.23 (-13)$&$    1.23 (-10)$&$    3.88 (-4)$&$    3.75 (-5)$&$    4.16 (-17)$&$    5.79 (-10)$\\
 $20$&$   1.99 (-15)$&$    2.10 (-15)$&$    4.59 (-8)$&$    1.34 (-13)$&$    2.22 (-15)$&$    6.86 (-6)$&$    1.47 (-10)$&$    1.94 (-16)$&$    4.99 (-16)$\\ 
 $30$&$   1.77 (-15)$&$    2.05 (-15)$&$    2.18 (-12)$&$    5.45 (-14)$&$    5.61 (-14)$&$    7.47 (-7)$&$    3.33 (-15)$&$    1.80 (-16)$&$    5.55 (-17)$\\
 $40$&$   5.10 (-15)$&$    1.24 (-14)$&$    2.54 (-14)$&$    2.92 (-13)$&$    3.78 (-13)$&$    5.97 (-8)$&$    1.99 (-15)$&$    1.80 (-16)$&$    3.88 (-16)$\\ 
 $50$&$   4.32 (-15)$&$    2.58 (-14)$&$    6.98 (-13)$&$    3.35 (-14)$&$    2.49 (-13)$&$    6.07 (-10)$&$    4.21 (-15)$&$    6.93 (-17)$&$             0 $ \\
 $60$&$   7.32 (-15)$&$    1.66 (-16)$&$    4.18 (-14)$&$    5.99 (-15)$&$    4.19 (-13)$&$    5.86 (-10)$&$    7.32 (-15)$&$    4.16 (-17)$&$    9.21 (-15)$\\ 
 $70$&$   7.21 (-15)$&$    1.08 (-14)$&$    1.01 (-14)$&$    6.70 (-13)$&$    9.43 (-13)$&$    3.50 (-11)$&$    7.32 (-15)$&$    3.46 (-16)$&$    3.49 (-15)$\\
 $80$&$   1.11 (-16)$&$    9.38 (-15)$&$    8.10 (-13)$&$    2.25 (-13)$&$    9.39 (-13)$&$    1.12 (-11)$&$    4.88 (-15)$&$    5.13 (-16)$&$    5.55 (-17)$\\
 $90$&$   1.22 (-15)$&$    3.36 (-14)$&$    2.44 (-14)$&$    4.53 (-13)$&$    9.39 (-13)$&$    8.46 (-13)$&$   2.37 (-14)$&$    1.38 (-17)$&$    3.74 (-14)$\\ 
 $100$&$  2.33 (-15)$&$    3.68 (-14)$&$    1.27 (-13)$&$    4.10 (-13)$&$    1.27 (-13)$&$    1.18 (-12)$&$    2.66 (-15)$&$             0           $&$  2.83 (-14)$ \\ 
\hline
\hline
$ I_2$&$ 2.17 (-13)$&$   1.11 (-16)$&$            0  $&$   \mbox{\tt  NaN}   $&$  \mbox {\tt  NaN}   $&$     0   $&$   \mbox{\tt  NaN}  $&$  5.96 (-13)$&$   5.55 (-17)$\\
\hline
$I_1$ & $8.74 (-10)$&$      0 $&$  1.11 (-16)$&$  \mbox{\tt  NaN}   $&$   \mbox{\tt  NaN}   $&$   2.77 (-17)$&$   \mbox {\tt  NaN}   $&$  1.38 (-17)$&$          0$ \\
\hline
\end{tabular}\end{table*}
\begin{table*}
  \caption{Absolute errors of the integrals, with weight $ w^{(2)}$,   computed by {\tt GaussMOP.m},  for $n=10, 20,\ldots, 100, $ and by $ I_1$ and $I_2$.}
	\centering
  \label{tab:ExNum1.1}
		\footnotesize
\begin{tabular}{@{}|@{\hspace{.4mm}}r@{\hspace{.4mm}}|@{\hspace{.4mm}}c@{\hspace{.4mm}}|@{\hspace{.4mm}}c@{\hspace{.4mm}}|@{\hspace{.4mm}}c@{\hspace{.4mm}}|@{\hspace{.4mm}}c@{\hspace{.4mm}}|@{\hspace{.4mm}}c@{\hspace{.4mm}}|@{\hspace{.4mm}}c@{\hspace{.4mm}}|@{\hspace{.4mm}}c@{\hspace{.4mm}}|@{\hspace{.4mm}}c@{\hspace{.4mm}}|@{\hspace{.4mm}}c@{\hspace{.4mm}}|@{}}%{|@{\hspace{.1cm}}l@{\hspace{.1cm}}|c|c|c|c|@{}c@{}|}
\hline
\multicolumn{10}{|c|}{$\int_{\Delta^{(i)}}f(x) w^{(2)}(x)dx, \quad f(x)= xe^{-x}$} \\
\hline
$n$ &$\mbox{\bf MOP}_1$ &$\mbox{\bf MOP}_2$&$\mbox{\bf MOP}_3$ &$\mbox{\bf MOP}_4$ &$\mbox{\bf MOP}_5$ &$\mbox{\bf MOP}_6$ &$\mbox{\bf MOP}_7$ &$\mbox{\bf MOP}_8$&$\mbox{\bf MOP}_9$ \\ 
\hline
  $ 10 $&$   1.11 (-16)$&$    2.35 (-8)$&$     2.33 (-3)$&$    5.32 (-15)$&$     3.103 (-11)$&$   1.97 (-3)$&$    1.21 (-3)$&$    1.11 (-16)$&$    2.64 (-10)$\\
  $ 20 $&$   1.99 (-15)$&$    1.05 (-15)$&$    7.19 (-7)$&$    1.37 (-13)$&$     1.93 (-14)$&$    4.61 (-5)$&$    3.90 (-9)$&$    2.22 (-16)$&$    4.44 (-16)$\\
  $ 30 $&$   2.10 (-15)$&$    1.11 (-15)$&$    1.64 (-10)$&$    8.31 (-14)$&$    3.41 (-14)$&$    6.85 (-7)$&$    2.22 (-16)$&$    3.60 (-16)$&$             0  $ \\
  $ 40 $&$   4.99 (-15)$&$    5.19 (-15)$&$    5.72 (-14)$&$    7.36 (-13)$&$    2.76 (-14)$&$    1.36 (-7)$&$    4.21 (-15)$&$    3.05 (-16)$&$    2.22 (-16)$\\
  $ 50 $&$   4.32 (-15)$&$    1.06 (-14)$&$    6.81 (-13)$&$    6.02 (-13)$&$    9.65 (-15)$&$    1.92 (-8)$&$    1.37 (-14)$&$    1.94 (-16)$&$    1.11 (-16)$\\ 
  $ 60 $&$   7.54 (-15)$&$    4.24 (-15)$&$    4.28 (-14)$&$    3.63 (-13)$&$    2.35 (-15)$&$    1.58 (-10)$&$    1.86 (-14)$&$    8.32 (-17)$&$   9.38 (-15)$\\ 
  $ 70 $&$   7.66 (-15)$&$    2.35 (-15)$&$    6.99 (-15)$&$    1.54 (-13)$&$    4.79 (-14)$&$    3.29 (-10)$&$    8.21 (-15)$&$    2.22 (-16)$&$    3.33 (-15)$\\ 
  $ 80 $&$            0         $&$    1.12 (-14)$&$    7.90 (-13)$&$    3.07 (-13)$&$    1.92 (-14)$&$    2.06 (-12)$&$    4.66 (-15)$&$    4.44 (-16)$&$    7.21 (-16)$\\ 
  $ 90 $&$   1.44 (-15)$&$    6.77 (-15)$&$    2.22 (-14)$&$    1.65 (-13)$&$    2.74 (-14)$&$    9.16 (-12)$&$    3.01 (-14)$&$    2.77 (-17)$&$   3.98 (-14)$\\
  $100 $&$   2.55 (-15)$&$    1.29 (-14)$&$    1.23 (-13)$&$    4.26 (-13)$&$    1.82 (-13)$&$    1.35 (-12)$&$    1.99 (-15)$&$    1.66 (-16)$&$    2.80 (-14)$\\ 
\hline
\hline
$I_2$ &$ 5.58 (-14)$&$             0 $&$             0 $&$           \mbox{\tt NaN}$&$   5.55 (-17)$&$   5.55 (-17)$&$           \mbox{\tt NaN}$&$   2.77 (-17)$&$   5.55 (-17)$\\
\hline
$I_1 $&$1.06 (-13)$&$   5.55 (-17)$&$             0 $&$           \mbox{\tt NaN}$&$             0 $&$             0 $&$           \mbox{\tt NaN}$&$   2.77 (-17)$&$             0 $\\
\hline
\end{tabular}\end{table*}
%%%%%%%%%%

\normalsize
%The results yielded by the {\tt Matlab} function {\tt integral} are the same of those of {\tt vpaintegral} are reported in Table

%We  observe that  while {\tt GaussMOP} has

\begin{table*}
  \caption{Executuion times in seconds required for computing  both  integrals, with weights $ w^{(1)}$ and  $ w^{(2)},$  by {\tt GaussMOP.m},  for $n=10,20,\ldots, 100, $ and by $ I_1$ and $I_2$.}
		\centering
  \label{tab:ExNum1.t}
		\footnotesize
\begin{tabular}{@{}|@{\hspace{.4mm}}r@{\hspace{.4mm}}|@{\hspace{.4mm}}c@{\hspace{.4mm}}|@{\hspace{.4mm}}c@{\hspace{.4mm}}|@{\hspace{.4mm}}c@{\hspace{.4mm}}|@{\hspace{.4mm}}c@{\hspace{.4mm}}|@{\hspace{.4mm}}c@{\hspace{.4mm}}|@{\hspace{.4mm}}c@{\hspace{.4mm}}|@{\hspace{.4mm}}c@{\hspace{.4mm}}|@{\hspace{.4mm}}c@{\hspace{.4mm}}|@{\hspace{.4mm}}c@{\hspace{.4mm}}|@{}}%{|@{\hspace{.1cm}}l@{\hspace{.1cm}}|c|c|c|c|@{}c@{}|}
\hline
\multicolumn{10}{|c|}{Execution  time in seconds} \\
\hline
$n$ &$\mbox{\bf MOP}_1$ &$\mbox{\bf MOP}_2$&$\mbox{\bf MOP}_3$ &$\mbox{\bf MOP}_4$ &$\mbox{\bf MOP}_5$ &$\mbox{\bf MOP}_6$ &$\mbox{\bf MOP}_7$ &$\mbox{\bf MOP}_8$&$\mbox{\bf MOP}_9$ \\ 
\hline
 $ 10 $&$   5.51 (-2)$&$ 1.00 (-2)$&$ 6.21 (-3)$&$ 4.47 (-3)$&$ 1.07 (-2)$&$ 4.11 (-3)$&$ 2.95 (-3)$&$ 5.77 (-3)$&$ 3.37 (-2)$\\
 $ 20 $&$   6.27 (-3)$&$ 2.00 (-3)$&$ 2.56 (-3)$&$ 2.01 (-3)$&$ 3.49 (-3)$&$ 1.55 (-3)$&$ 1.44 (-3)$&$ 2.38 (-3)$&$ 2.85 (-3)$\\
 $ 30 $&$   1.94 (-3)$&$ 1.93 (-3)$&$ 1.83 (-3)$&$ 1.89 (-3)$&$ 2.33 (-3)$&$ 1.85 (-3)$&$ 1.94 (-3)$&$ 2.00 (-3)$&$ 2.15 (-3)$\\
 $ 40 $&$   3.00 (-3)$&$ 2.97 (-3)$&$ 2.93 (-3)$&$ 3.03 (-3)$&$ 2.89 (-3)$&$ 2.91 (-3)$&$ 2.83 (-3)$&$ 2.99 (-3)$&$ 3.00 (-3)$\\
 $ 50 $&$   6.38 (-3)$&$ 5.95 (-3)$&$ 6.01 (-3)$&$ 6.14 (-3)$&$ 7.02 (-3)$&$ 6.98 (-3)$&$ 6.32 (-3)$&$ 6.65 (-3)$&$ 7.29 (-3)$\\
 $ 60 $&$   5.78 (-3)$&$ 5.45 (-3)$&$ 5.23 (-3)$&$ 5.53 (-3)$&$ 5.16 (-3)$&$ 4.89 (-3)$&$ 4.96 (-3)$&$ 5.06 (-3)$&$ 5.69 (-3)$\\
 $ 70 $&$   6.71 (-3)$&$ 6.60 (-3)$&$ 6.61 (-3)$&$ 6.72 (-3)$&$ 6.58 (-3)$&$ 6.52 (-3)$&$ 6.52 (-3)$&$ 6.63 (-3)$&$ 7.44 (-3)$\\
 $ 80 $&$   8.68 (-3)$&$ 8.31 (-3)$&$ 8.40 (-3)$&$ 8.34 (-3)$&$ 8.34 (-3)$&$ 8.31 (-3)$&$ 8.54 (-3)$&$ 8.55 (-3)$&$ 9.23 (-3)$\\
 $ 90 $&$   1.08 (-2)$&$ 1.07 (-2)$&$ 1.08 (-2)$&$ 1.06 (-2)$&$ 1.06 (-2)$&$ 1.04 (-2)$&$ 1.05 (-2)$&$ 1.07 (-2)$&$ 1.06 (-2)$\\
 $ 100 $&$  1.31 (-2)$&$ 1.29 (-2)$&$ 1.28 (-2)$&$ 1.27 (-2)$&$ 1.28 (-2)$&$ 1.26 (-2)$&$ 1.27 (-2)$&$ 1.29 (-2)$&$ 1.29 (-2)$\\
\hline
\hline 
$ I_1 $&$  2.68 (-2)$&$ 1.10 (-2)$&$ 5.64 (-3)$&$ 5.12 (-3)$&$ 5.65 (-3)$&$ 3.28 (-3)$&$ 3.96 (-3)$&$ 2.49 (-1)$&$ 2.61 (0)$\\
\hline
 $ I_2 $&$  3.11 (-3)$&$ 2.31 (-3)$&$ 9.91 (-4)$&$ 2.26 (-3)$&$ 1.79 (-3)$&$ 7.18 (-4)$&$ 1.60 (-3)$&$ 9.82 (-2)$&$ 1.39 (0)$\\
\hline
\end{tabular}\end{table*}

In Table~\ref{tab:ExNum1.t} the execution times (in seconds) required by   {\tt GaussMOP.m}, for different values of $ n$, and by the {\tt Matlab} functions $ I_1$ and $I_2 $, % {\tt integral}
applied for computing both  integrals with  weights $w^{(1)} $ and $ w^{(2)}$,   are reported. Although the code {\tt GaussMOP.m} is interpreted by {\tt Matlab},  {\tt GaussMOP.m}, $I_1$ and $I_2$, exhibit comparable execution times,
except for $\mbox{\bf MOP}_8$ and $\mbox{\bf MOP}_9$, for which  {\tt GaussMOP.m} is significantly faster than $ I_1$ and $I_2$.
 %for the first $ 7 $ MOPs, while the  proposed method is much faster than $ I_1$ and $I_2$  for $\mbox{\bf MOP}_8$ and $\mbox{\bf MOP}_9$.  

\end{example}

\section{Conclusions}
\label{sect:conc}
A {\tt Matlab} package called {\tt GaussMOP.m} is proposed for computing simultaneous Gaussian quadrature rules associated with different kinds of  MOPs. The nodes and weights of such rules are retrieved from the   eigendecomposition of a banded lower Hessenberg matrix, which turns out to be an  ill--conditioned eigenvalue problem.
Making use of a novel balancing procedure, the eigenvalue condition of the latter Hessenberg matrix is drastically reduced. 
Moreover, a variant of 
 the Aberth--Ehrlich method  is used to compute  the eigenvalues and  associated  left and right eigenvectors with $ {\mathcal O}(n)$ memory and $ {\mathcal O}(n^2)$  computational complexity.

{\tt GaussMOP.m} was applied for simultaneously computing integrals with two different weights associated with the considered MOPs, and its performance was compared, in terms of accuracy, robustness and execution time, to that of the {\tt Matlab} intrinsic  function {\tt integral.m}.  The values obtained by computing the latter integrals by the  {\tt Mathematica} function {\tt Integrate}, requiring  a precision of  $80 $ digits, were considered as the exact ones.

The numerical tests show the reliability of the proposed numerical method.

%We observed that, in some cases, the function {\tt integral.m} fails to compute the considered integrals and its execution times are significantly higher than those required by {\tt GaussMOP.m}. 
%Concerning the comparison between {\tt GaussMOP.m} and the  {Mathematica} function {\tt NIntegrate}, the numerical tests highlighted 
%that the numerical approximations yielded by {\tt NIntegrate} are not computed to full double precision accuracy.

\section*{Acknowledgements}

%\begin{acks}
Teresa Laudadio and Nicola Mastronardi are members of the Gruppo Nazionale Calcolo Scientifico-Istituto Nazionale di Alta Matematica (GNCS-INdAM).
The work of Nicola Mastronardi  was partly supported by MIUR, 
PROGETTO DI RICERCA DI
RILEVANTE INTERESSE NAZIONALE (PRIN)
20227PCCKZ
``Low--rank Structures and Numerical Methods in Matrix and Tensor
Computations and their Application'',
Universit\`a degli Studi di BOLOGNA CUP
J53D23003620006. 
The work of Walter Van Assche was supported by FWO grant G0C9819N.
The work of Paul Van Dooren  was partly supported by Consiglio Nazionale delle Ricerche, under the Short Term Mobility program. 
%\end{acks}


 \bibliographystyle{elsarticle-num} 
 \bibliography{MOPbib1}


%\begin{thebibliography}{00}


%\bibitem{}

\newpage
%\end{thebibliography}
\section{Appendix: {\tt Matlab} codes} \label{sect:append}
The {\tt Matlab } functions implementing the simultaneus Gaussian quadrature rules for MOPs are displayed  in the following Tables. The whole package can be downloaded from \\
{\tt https://users.ba.cnr.it/iac/irmanm21/MOP/}

\begin{table}[h]
  \caption{Matlab function   {\tt GaussMOP.m}}
  \label{tab:Alg0}
%\begin{enumerate}[label={\bf A.\arabic*}]
  {\tt function[x,w1,w2,ier]=GaussMOP(b,c,d,n,F)}  \\
  {\tt \% compute the nodes and weights of the simultaneous }  \\
  {\tt \% Gaussian quadrature rule associated with MOPs}  \\
	 {\tt \%  Input:}\\
 {\tt   b=diag(}$ {H}_{n}${\tt );}\\\
 {\tt    c =diag(}$ {H}_{n}${\tt );}\\\
 {\tt    d=diag(}$ {H}_{n},-2${\tt );}\\
 {\tt    n : size of} $ \hat{H}_n$\\
 {\tt    F:} $ 2\times 2 $ {\tt matrix used for computing the weights}\\
 {\tt \% Output:}\\
 {\tt \%  x, w1, w2: nodes and weights of the simultaneous Gauss rule}\\
{\tt \%  ier: ier is set to zero for normal return, ier is set to j if }\\
{\tt \% the j-th eigenvalue has not been determined after 30 iterations.}\\
 % {\tt \% compute the entries of the scaled matrix}  \\
	{\tt a=ones(1,n+1);}\\
	{\tt \%balancing the Hessenberg matrix }\\
  {\tt [as,bs,cs,ds,S] = DScaleS2(a,b,c,d,n+1);}  \\
  {\tt \% transformation of the Hessenberg matrix into a similar nonsymmetric}  \\
	 {\tt \% tridiagonal matrix  by means of elementary transformations }  \\
  {\tt [GEL,as1,bs1,cs1] =  tridEHbackwardV(as,bs,cs,ds,n);}  \\
  {\tt \% transformation of the nonsymmetric tridiagonal matrix}\\
	{\tt \% into a similar symmetric tridiagonal matrix}  \\
  {\tt  [bs2,cs2] = DScaleSV2(as1,bs1,cs1,n+1);}  \\
  {\tt \% computation of the eigenvalues of the symmetric tridiagonal matrix}  \\
  {\tt [x,e1,z1,ierr]= gausq2(n, bs2(1:n), cs2(2:n));}  \\
  {\tt x=sort(x);}  \\
  {\tt \% computation of the nodes and weights}  \\
  {\tt \% by means of the Ehrlich-Aberth Method}  \\
  {\tt tol=n\^{}2*eps; }  \\
  {\tt [x,w1,w2,ier] = EA\_method(as,bs,cs,ds,x,tol,F,S,n);}  \\
 {\tt end }
%\end{enumerate}
\end{table}
%%%%%%%%%%%%%%%%%%%%%%%%%%%%%%%%%%%%%%%%%%%%%%%%%%%%%%%%%%%%%%%%%
\begin{table}[h]
  \caption{Matlab function   {\tt DScaleS2.m}}
  \label{tab:Alg1}
%\begin{enumerate}[label={\bf B.\arabic*}]
 \begin{tabbing}
 {\tt func}\={\tt tion}\={\tt  [a,b,c,d,s] = DScaleS2(a,b,c,d,n) } \\
{\tt \% Diagonal scaling of the  banded Hessenberg matrix} $H_n$\\
{\tt \%   has as input a lower Hessenberg matrix H with bandwidth 4 }\\ 
{\tt \%   and returns a scaled version} $ \hat{H}_n = S_n^{-1}{H}_n S_n $\\ 
{\tt \%   where} $S_n$ {\tt  is a diagonal matrix}\\ 
{\tt \% Input:}\\
{\tt \%   a =diag(}$ {H}_n${\tt ,1);}\\
{\tt \%   b=diag(}$ {H}_n${\tt );}\\
{\tt \%   c=diag(}$ {H}_n${\tt ,-1);}\\
{\tt \%   d=diag}$ {H}_n${\tt ,-2);}\\
{\tt \%   n : size of} $ {H}_n$;\\
{\tt \%   Output:}\\
{\tt \%       a=diag(}$ \hat{H}_n${\tt ,1);}\\
{\tt \%       b=diag(}$ \hat{H}_n${\tt );}\\
{\tt \%       c=diag(}$ \hat{H}_n${\tt ,-1);}\\
{\tt \%       d=diag(}$ \hat{H}_n${\tt ,-2);}\\
{\tt \%       s=diag(}$ {S}_n${\tt,1);}\\
 {\tt s(1)=1;s(2)=sqrt(c(2)/a(1));  }\\
 {\tt t1=sqrt(a(1)/c(2));  }\\
 {\tt a(1)=a(1)/t1; \;c(2)=a(1);   } \\
% {\tt c(2)=a(1);  }\label{A1.5} 
 {\tt for i=2:n-1  } \\
\>{\tt t2=sqrt(a(i)/c(i+1));\; d(i+1)=d(i+1)*t1*t2;    }\\
% \>{\tt }\\
 \>{\tt a(i)=sqrt(a(i)*c(i+1));\;  c(i+1)=a(i); t1=t2;     } \\
% {\tt   } \\
% {\tt \;\;\; d(i+1)=d(i+1)*s(i-1)/s(i+1);   }\label{A1.10} 
 {\tt end }
\end{tabbing}
\end{table}
%%%%%%%%%%%%%%%%%%%%%%%%%%%%%%%%%%%%%%%%%%%%%%%%%%%%%%%%%%%%%%%%%
\begin{table}[h]
  \caption{Matlab function   {\tt DScaleSV2.m}}
  \label{tab:Alg1.0}
%\begin{enumerate}[label={\bf C.\arabic*}]
\begin{tabbing}
 {\tt func}\={\tt tion}\={\tt  [b,c] = DScaleSV2(a,b,c,n) }\\
 {\tt \% transformation of the unsymmetric tridiagonal matrix Tof size n}\\
 {\tt \% into  a similar  symmetric tridiagonal matrix};\\
{\tt \% Input:}\\
{\tt \%       a=diag(}$ \hat{T}_n${\tt ,1);}\\
{\tt \%       b=diag(}$ \hat{T}_n${\tt );}\\
{\tt \%       c=diag(}$ \hat{T}_n${\tt ,-1);}\\
{\tt \%       n : size of} $ \hat{T}_n$;\\
{\tt \%       Output:}\\
{\tt \%       b=diag(}$ {T}_n${\tt );}\\
{\tt \%       c=diag(}$ {T}_n${\tt ,-1);}\\
 {\tt for i=1:n-1  }\\
\> {\tt a(i)=sqrt(a(i)*c(i+1));     }\\
 \>{\tt c(i+1)=a(i); }\\ 
 {\tt end }
\end{tabbing}
\end{table}
\begin{table*}[h]
  \caption{{\tt Matlab} function   {\tt tridEHbackwardV.m}}
  \label{tab:trid}
%\begin{enumerate}[label={\bf D.\arabic*}]
\begin{tabbing}
 {\tt func}\={\tt tio}\={\tt n[a,b,c]}\= ={\tt tridEHbackwardV(a,b,c,d,n) }\\
{\tt \%   reduction of the lower Hessenberg matrix} $ \hat{H}_n$\\
{\tt \%   into the similar  similar tridiagonal one} $ \hat{T}_n$\\
{\tt \%   Input: }\\
{\tt \%   a =diag(}$ \hat{H}_n,1${\tt );}\\
{\tt \%   b=diag(}$ \hat{H}_n${\tt );}\\
{\tt \%   c =diag(}$ \hat{H}_n,-1${\tt );}\\
{\tt \%   d=diag(}$ \hat{H}_n,-2${\tt );}\\
{\tt \%   n : size of }$ \hat{H}_n$;\\
{\tt \%   Output:}\\
{\tt \%   a =diag(}$ \hat{T}_n,1${\tt );}\\
{\tt \%   b=diag(}$ \hat{T}_n${\tt );}\\
{\tt \%   c =diag(}$ \hat{T}_n,-1${\tt );}\\
% {\tt \% transformation of $ \hat{H}_n$  into  $ \hat{T}_n$ }\\
 {\tt for i=n:-1:3  } \\
 \>{\tt h=d(i)/c(i);     } \\
 \>{\tt d(i)=0;\;\;c(i-1)=c(i-1)-h*b(i-1);\; b(i-2)=b(i-2)-h*a(i-2);}\\
%\>{\tt  } \\
 \>{\tt t=d(i-2)*h;} \\% ;\;
 \> {\tt d(i-1)=d(i-1)+h*c(i-2);\;\;c(i-1)=c(i-1)+h*b(i-2);\;b(i-1)=b(i-1)+h*a(i-2);}\\
%\>{\tt }\\
 \>{\tt for j=i-1:-2:4}\\
 \>\>{\tt h=t/d(j);}\\ %\;
\>\> {\tt d(j-1)=d(j-1)-h*c(j-1);\;c(j-2)=c(j-2)-h*b(j-2);\; b(j-3)=b(j-3)-h*a(j-3);}\\
%\>\>{\tt}\\
\>\> {\tt t=d(j-3)*h;}\\
 \>\>{\tt d(j-2)=d(j-2)+h*c(j-3);\;c(j-2)=c(j-2)+h*b(j-3);\; b(j-2)=b(j-2)+h*a(j-3);}\\
%\>\>{\tt}\\
		\> {\tt end }	\\	
 {\tt end }
\end{tabbing}
\end{table*}
%%%%%%%%%%%%%%%%%%%%%%%%%%%%%%%%%%%%%%%%%%%%%%%%%%%%%%%%%%%%%%%%%%
\small
\begin{table*}[h]
  \caption{{\tt Matlab} function   {\tt EA\_method.m}}
  \label{tab:AlgEA}
%\begin{enumerate}[label={\bf E.\arabic*}]
 \begin{tabbing}
 {\tt func}\={\tt tion}\={\tt [x,}\={\tt w1, w2,}\={\tt ier}\={\tt \; = EA\_method(a,b,c,d,x,tol,F,S,n)}\\
  {\tt \%  Ehrlich-Aberth Method for computing the zeros  and the left }\\
  {\tt \% and right eigenvector of the Hessenberg matrix  associated to MOPs}\\
 {\tt \%  and the nodes and weights of the simultaneous Gauss rule}\\
 {\tt \%  Input:}\;
 {\tt    a =diag(}$ \hat{H}_{n,n+1}${\tt );}\;\;
 {\tt   b=diag(}$ \hat{H}_{n,n+1}${\tt );}\;\;
 {\tt    c =diag(}$ \hat{H}_{n,n+1}${\tt );}\;\;
 {\tt    d=diag(}$ \hat{H}_{n,n+1},-2${\tt );}\\
 {\tt    x:  approximation of eigenvalues of} $ \hat{H}_n$\;\;
 {\tt    tol: tolerance}\;\;
 {\tt    n : size of} $ \hat{H}_n$\\
 {\tt \%   S: first two diagonal entries of} $ S_n$\;\;
 {\tt    F:} $ 2\times 2 $ {\tt matrix used for computing the weights}\\
 {\tt \% Output:}\\
 {\tt \%  x, w1, w2: nodes and weights of the simultaneous Gauss rule}\\
 %{\tt \%  weights of the simultaneous Gauss rule}\\
 {\tt \%  ier: ier is set to zero for normal return, ier is set to j if the j-th eigenvalue}\\
{\tt \% has not been determined after 30 iterations.}\\
  {\tt x0=x; \;}{\tt m=n; n=n+1;}\;
  {\tt conv=zeros(m,1);}\;
  {\tt iter=0; ier=0;}\;
  {\tt iterj=zeros(m,1);}\\
  {\tt   while sum(conv)<m \& ier }$\sim${\tt =0,}\\
\> {\tt        iter=iter+1;}\\
\>{\tt       for j=1:m}\\
\>\> {\tt                x0=x(j);}\\
\>\> {\tt               if conv(j)==0,}\\
\>\>\> {\tt                [p0,p1] = one\_step\_Newton(a,b,c,d,x0,n);}\\
\>\>\> {\tt                   p01=p0(n)/p1(m);}\\
\>\>\> {\tt                         sum1=0;}\\
\>\>\> {\tt                        for k=1:m}\\
\>\>\>\>{\tt                                if k}$\sim${\tt =j}\\
\>\>\>\>\> {\tt                                  sum1=sum1+1/(x(j)-x(k));}\\
\>\>\>\> {\tt                                  end \% if}\\
\>\>\>{\tt                          end \% for}\\
\>\>\> {\tt              t=x(j)-p01/(1-p01*sum1);}\\
\>\>\> {\tt             x(j)=t;}\\
\>\>\> {\tt              if abs(p01)< tol | abs(p0(n)) < tol,}\\
\>\>\>\> {\tt                               conv(j)=1;}\\
{\tt \% computation of the right eigenvector} $v$ {\tt associated with} $ x_j$\\
\>\>\> {\tt                 v=p0(1:m);}\\
{\tt \%  computation of the left eigenvector} $u$ {\tt associated with} $ x_j$ \\
\>\>\> {\tt                 [u]=left\_eigvect(x(j),a,b,c,d,m);}\\
\>\>\> {\tt                   w1(j)=v(1)*(F(1,1)*u(1))/dot(u,v);}\\
\>\>\>{\tt                  w2(j)=S(1)*v(1)*(F(2,1)*u(1)/S(1)+F(2,2)*u(2)/S(2))/dot(u,v);}\\
\>\>\> {\tt                  iterj(j)=iter;}\\
\>\>{\tt                        else}\\
\>\>\> {\tt                            x0=x(j);}\\ 
\>\>{\tt                       end \% if}\\
\> {\tt                  end \%if} \\
\> {\tt         if iter == 30,}\\
\>\> {\tt         ier=j;} \\
\>{\tt        end \%if} \\
 {\tt end \% while }
\end{tabbing}
\end{table*}
%%%%%%%%%%%%%%%%%%%%%%%%%%%%%%%%%%%%%%%%%%%%%%%%%%%%%%%%%%%%%%%%%%%
\begin{table*}[h]
  \caption{{\tt Matlab} function   {\tt one\_step\_Newton.m}}
  \label{tab:AlgG}
%\begin{enumerate}[label={\bf F.\arabic*}]
  \begin{tabbing}
 {\tt func}\={\tt tion}\={\tt   [p,p1] = one\_step\_Newton(a,b,c,d,x,n)}\\
  {\tt \% computation of the sequence of  MOPs and their derivatives evaluated in} $ x$\\
	 {\tt \% Input:}\\
 {\tt \%   a =diag(}$ \hat{H}_{n,n+1},1$,{\tt );}\\
 {\tt \%   b=diag(}$ \hat{H}_{n,n+1}${\tt );}\\
 {\tt \%   c =diag(}$ \hat{H}_{n,n+1},-1${\tt );}\\
 {\tt \%   d=diag(}$ \hat{H}_{n,n+1},-2${\tt );}\\
 {\tt \%   x:  approximation of an eigenvalue of} $ \hat{H}_n$\\
 {\tt \%   n : size of} $ \hat{H}_n$\\
 {\tt \%   Output:}\\
 {\tt \%   p, vector} $ [p_0(\lambda), p_1(\lambda),\ldots,p_n(\lambda)]^T$\\
 {\tt \%   p10, vector} $ [p'_1(\lambda), p'_2(\lambda),\ldots,p'_n(\lambda)]^T$\\
 % {\tt \% and their derivatives evaluated in x}\\
  {\tt m=n; \;\;n=n+1;\;\;b=b-x;\;\;a0=a;\;\;b0=b;\;\;c0=c;\;\;d0=d;\;\; p=zeros(n,1);}\\
 % {\tt b=b-x;\;\;a0=a;\;\;b0=b;\;\;c0=c;\;\;d0=d;\;\; p=zeros(n,1); }\\
 % {\tt}\\
 {\tt \% transformation of the} $ \hat{H}_{n,n+1}$  {\tt  into a lower tridiagonal matrix}\\
 {\tt \% by means of the multiplication of} $ n $  {\tt  Givens rotations to the right}\\
 %{\tt \% in order to annihilate the first superdiagonal of } $ \hat{H}_{n,n+1}$\\
  {\tt for i=1:m-3}\\
 \> {\tt  G=givens\_GV(b(i),a(i));\;\;CS1(i,1)=G(1,1);\;\;CS1(i,2)=G(1,2);}\\
 \> {\tt  V1=[b(i),a(i)]*G';\;\;b(i)=V1(1);\;\;a(i)=V1(2);}\\
 \> {\tt  V1=[c(i+1),b(i+1)]*G';\;\;c(i+1)=V1(1);\;\;b(i+1)=V1(2);}\\
 \> {\tt  V1=[d(i+2),c(i+2)]*G';\;\;d(i+2)=V1(1);\;\;c(i+2)=V1(2); }\\
 \> {\tt  V1=[0,d(i+3)]*G';\;\;d(i+3)=V1(2);}\\
  {\tt end}\\
  {\tt i=m-2;\;\;  G=givens\_GV(b(i),a(i));\;\;CS1(i,1)=G(1,1);\;\;CS1(i,2)=G(1,2);}\\
 % {\tt}\\
  {\tt V1=[b(i),a(i)]*G';\;\;b(i)=V1(1);\;\;a(i)=V1(2);}\\
  {\tt V1=[c(i+1),b(i+1)]*G';\;\;c(i+1)=V1(1);\;\;b(i+1)=V1(2);}\\
  {\tt V1=[d(i+2),c(i+2)]*G';\;\;d(i+2)=V1(1);\;\;c(i+2)=V1(2);}\\
  {\tt i=m-1;\;\;G=givens\_GV(b(i),a(i));\;\;CS1(i,1)=G(1,1);\;\;CS1(i,2)=G(1,2);}\\
 %{\tt }\\
  {\tt V1=[b(i),a(i)]*G';\;\;b(i)=V1(1);\;\;a(i)=V1(2);}\\
  {\tt V1=[c(i+1),b(i+1)]*G';\;\;c(i+1)=V1(1);\;\;b(i+1)=V1(2);}\\
 
  {\tt i=m;\;\;G=givens\_GV(b(i),a(i));\;\;CS1(i,1)=G(1,1);\;\;CS1(i,2)=G(1,2);}\\
 % {\tt }\\
  {\tt V1=[b(i),a(i)]*G';\;\;b(i)=V1(1);\;\;a(i)=V1(2);}\\
  {\tt p(n)=CS1(n-1,1);\;\;\tt t=1;}\\
 % {\tt t=1;}\\
  {\tt for i=n-1:-1:2,}\\
 \> {\tt  t=-t*CS1(i,2);\;\;p(i)=CS1(i-1,1)*t;}\\
   %{\tt  p(i)=CS1(i-1,1)*t;}\\
  {\tt end}\\
  {\tt p(1)=-t*CS1(1,2);}\\
  %{\tt \% computation of the sequence of the derivative of the orthogonal}\\
  %{\tt \% polynomials evaluated in x}\\
	 {\tt \% computation of } $\hat{p}_j^{'}(x), \; j=1,\ldots,n$\\
  {\tt p1=zeros(m,1);}\\
  {\tt p1(1)=p(1)/a0(1);\;\;p1(2)=(p(2)-b0(2)*p1(1))/a0(2);}\\
  %{\tt p1(2)=(p(2)-b0(2)*p1(1))/a0(2);}\\
  {\tt p1(3)=(p(3)-c0(3)*p1(1)-b0(3)*p1(2))/a0(3);}\\
  {\tt for i=4:m}\\
 \> {\tt  p1(i)=(p(i)-d0(i)*p1(i-3)-c0(i)*p1(i-2)-b0(i)*p1(i-1))/a0(i);}\\
  {\tt end}
\end{tabbing}
\end{table*}

%%%%%%%%%%%%%%%%%%%%%%%%%%%%%%%%%%%%%%%%%%%%%%%%%%%%%%%%%%%%%%%%%%%
\begin{table*}[h]
  \caption{{\tt Matlab} function   {\tt left\_eigvect.m}}
  \label{tab:eigvect}
\begin{tabbing}%[label={\bf G.\arabic*}]
  {\tt func}\={\tt tion}\={\tt   [u1]=left\_eigvect(x,a,b,c,d,n);}\\
  {\tt \% computation of the left eigenvector of $ \hat H_n$ corresponding to the eigenvalue $ x$}\\
	 {\tt \% Input:}\\
 {\tt \%   a=diag(}$ \hat{H}_{n},1$,{\tt );}\\
 {\tt \%   b=diag(}$ \hat{H}_{n}${\tt );}\\
 {\tt \%   c=diag(}$ \hat{H}_{n},-1${\tt );}\\
 {\tt \%   d=diag(}$ \hat{H}_{n},-2${\tt );}\\
 {\tt \%   x: eigenvalue of} $ \hat{H}_n$\\
 {\tt \%   n: size of} $ \hat{H}_n$\\
 {\tt \%   Output:}\\
{\tt \%  u, the left eigenvector of }$ \hat{H}_{n}$ {\tt  corresponding to the eigenvalue $ x$}\\
  {\tt b=b-x;}\\
  {\tt for i=n-1:-1:3}\\
  \>{\tt    G=givens\_GV(b(i+1),a(i));}\\
  \>{\tt     cs(i,1)=G(1,1);\;\;cs(i,2)=G(1,2);}\\
  \>{\tt     V=G'*[d(i) c(i) b(i) a(i); 0 d(i+1) c(i+1) b(i+1)];}\\
 \>{\tt     d(i)=V(1,1);\;\;c(i)=V(1,2);\;\;b(i)=V(1,3);\;\;a(i)=V(1,4);}\\
 \> {\tt     d(i+1)=V(2,2);\;\;c(i+1)=V(2,3);\;\;b(i+1)=V(2,4);}\\
  {\tt end}\\
  {\tt i=2;}\\
  {\tt G=givens\_GV(b(i+1),a(i));}\\
  {\tt cs(i,1)=G(1,1);\;\;cs(i,2)=G(1,2);}\\
  {\tt V=G'*[c(i) b(i) a(i); d(i+1) c(i+1) b(i+1)];}\\
  {\tt c(i)=V(1,1);\;\;b(i)=V(1,2);\;\;a(i)=V(1,3);}\\
  {\tt d(i+1)=V(2,1);\;\;c(i+1)=V(2,2);\;\;b(i+1)=V(2,3);}\\
  {\tt i=1;}\\
  {\tt G=givens\_GV(b(i+1),a(i));}\\
  {\tt cs(i,1)=G(1,1);\;\;cs(i,2)=G(1,2);}\\
  {\tt V=G'*[b(i) a(i);c(i+1) b(i+1)];}\\
  {\tt b(i)=V(1,1);\;\;a(i)=V(1,2);}\\
  {\tt c(i+1)=V(2,1);\;\;b(i+1)=V(2,2);}\\
  {\tt t=1; u=zeros(n,1);}\\
  {\tt for i=1:n-1}\\
\>{\tt  u(i)=t*cs(i,1);}\\
\>{\tt  t=-t*cs(i,2);}\\
  { \tt end}\\
  { \tt u(n)=t;}
\end{tabbing}
\end{table*}

%%%%%%%%%%%%%%%%%%%%%%%%%%%%%%%%%%%%%%%%%%%%%%%%%%%%%%%%%%%%%%%%%%%
\begin{table*}[h]
  \caption{{\tt Matlab} function   {\tt givens\_GV.m}}
  \label{tab:GV}
\begin{tabbing}%[label={\bf G.\arabic*}]
  {\tt func}\={\tt tion}\={\tt   [u1]=}\={\tt givens\_GV(x,a,b);}\\
  {\tt \%  computation of the Givens rotation}  $G=[c\; \;s;-s \;\; c];$ \\
{\tt	\%  such that  $  G[a;b]=[\sqrt{a^2+b^2};\;0]$}\\
  {\tt if b == 0}\\
	{\tt c=1; \; s=0; }\\
	 \>{\tt else}\\
	 \>\>{\tt if abs(b) > abs(a)} \\
 \>\>\> {\tt tau = -a/b;}\\
\>\>\> {\tt s = 1/sqrt(1+tau\^{}2);}\\
\>\>\> {\tt c = s * tau;}\\
\>\> {\tt else} \\
\>\>\> {\tt tau = -b/a;}\\
\>\>\> {\tt c = 1/sqrt(1+tau\^{}2); }\\
\>\>\> {\tt s = c * tau;}\\
\>\> {\tt end} \\
\> {\tt end} \\
\>{\tt    G=[c -s; s  c] ;}\\
  { \tt end}\\
\end{tabbing}
\end{table*}

\end{document}